\documentclass[12pt, leqno]{amsart}

\usepackage{amsfonts,amsthm,amsmath,xcolor,comment}
\numberwithin{equation}{section}
\usepackage{chngcntr}
\counterwithin{table}{section}

\usepackage{amssymb}

\theoremstyle{plain}
\newtheorem{thm}{Theorem}[section]

\newtheorem{lem}[thm]{Lemma}
\newtheorem{cor}[thm]{Corollary}

\theoremstyle{definition}
\newtheorem{df}{Definition}[section]
\newtheorem{rem}{Remark}[section]
\newtheorem{ex}{Example}[section]

\newcommand{\FF}{\mathbb{F}}
\newcommand{\ZZ}{\mathbb{Z}}

\newcommand{\CC}{\mathbb{C}}

\newcommand{\D}{\mathcal{D}}
\newcommand{\CJ}{\mathrm{CJ}}

\newcommand{\SCJ}{\mathrm{SCJ}}
\newcommand{\scwe}{\mathbf{scwe}}

\newcommand{\cwe}{\mathbf{cwe}}

\newcommand{\III}{\mathrm{III}}
\newcommand{\IV}{\mathrm{IV}}

\def\bm#1{\mathbf{#1}}

\DeclareMathOperator{\wt}{wt}

\DeclareMathOperator{\comp}{comp}

\begin{document}


\title{Jacobi polynomials and design theory II}

\author[Chakraborty]{Himadri Shekhar Chakraborty}
\address{Department of Mathematics, Shahjalal University of Science and Technology, Sylhet 3114, BANGLADESH}
\email{himadri-mat@sust.edu}

\author[Ishikawa]{Reina Ishikawa}
\address{Graduate School of Science and Engineering, Waseda University, Tokyo 169-8555, Japan}
\email{reina.i@suou.waseda.jp}
 

\author[Tanaka]{Yuuho Tanaka$\ast$}
\address{Graduate School of Science and Engineering, Waseda University, Tokyo, 169-8555, Japan}
\email{tanaka\_yuuho\_dc@akane.waseda.jp}

\thanks {*Corresponding author}

\date{}
\maketitle

\begin{abstract}
	In this paper, 
	we introduce some new polynomials associated to linear codes
	over~$\FF_{q}$. In particular, we introduce the 
	notion of split complete Jacobi polynomials 
	attached to multiple sets of coordinate places of a 
	linear code over~$\FF_{q}$, 
	and give the MacWilliams type identity for it.
	We also give the notion of generalized $q$-colored $t$-designs. 
	As an application of the generalized $q$-colored $t$-designs, 
	we derive a formula that obtains 
	the split complete Jacobi polynomials of a linear code over~$\FF_{q}$.
	Moreover, we define the concept of colored packing (resp. covering) designs.
	Finally, we give some coding theoretical applications 
	of the colored designs for Type~III and Type~IV codes.	
\end{abstract}

{\small
	\noindent
	{\bfseries Key Words:}
	Codes, Jacobi polynomials, designs, invariant theory.\\ \vspace{-0.15in}
	
	\noindent
	2010 {\it Mathematics Subject Classification}. 
	Primary 11T71;
	Secondary 94B05, 11F11.\\ \quad
}

\section{Introduction}

In 1997, Ozeki~\cite{Ozeki} introduced the notion of Jacobi polynomials of 
linear codes, an analogue to Jacobi forms~\cite{EZ} of lattices.
Many authors studied the Jacobi polynomials in coding theory;
for instance~\cite{BMS1999,BRS2000,BSU2001,CM2021,CMO2022, HOO}. 
Among these articles Bonnecaze et al.~\cite{BMS1999,BRS2000,BSU2001} pointed out some 
characterizations of the Jacobi polynomials with the codes supporting designs. 
Moreover, Bonnecaze, Rains and Solé~\cite{BRS2000} introduced the notion of 
colored $t$-designs 
and gave an application of these designs for $\ZZ_{4}$-codes in the 
evaluation of the Jacobi polynomials from the symmetrized weight enumerator 
of $\ZZ_{4}$-codes using the polarization operator. 
Later, Bonnecaze, Solé and Udaya~\cite{BSU2001} studied the $3$-colored $3$-designs
in the case of Type~$\III$ codes. 
Furthermore, Cameron~\cite{Cameron2009} gave a new generalization 
of the combinatorial $t$-designs. In this paper, we would like to call these 
designs as the generalized $t$-designs. In a recent study, 
Chakraborty, Miezaki, Oura and Tanaka~\cite{CMOTxxxx} 
introduced the notion of Jacobi polynomials of a linear code with 
multiple reference vectors and gave a design theoretical application of these Jacobi polynomials in the study of generalized $t$-designs.

In this paper, 
we introduced the notion of split complete weight enumerators 
for the linear codes over $\FF_{q}$ which is independent from 
the sense of Bonnecaze et al.~\cite{BSU2001}. 
We also introduce the notion of 
split complete Jacobi polynomials of a linear code over $\FF_{q}$ 
attached to multiple sets of coordinate places of the code.
We show that both the code polynomials: split complete weight enumerators 
and split complete Jacobi polynomials satisfy the MacWilliams type identities.
In particular, 
the complete Jacobi polynomials of linear codes over $\FF_{3}$ 
in our sense are equivalent to
the split complete weight enumerators of codes over~$\FF_{3}$ 
in the sense of Bonnecaze et al.~\cite{BSU2001}.
Moreover, we define the concept of the generalized colored $t$-designs,
and as an analogue to Bonnecaze et al.~\cite{BSU2001},
we present a combinatorial interpretation of the polarization of the 
split complete Jacobi polynomials of a linear code over~$\FF_{q}$.
In addition, we study the complete Jacobi polynomials of some
Type~$\III$ (resp. Type~$\IV$) codes of specific lengths 
through invariant theory to construct the colored packing 
(resp. covering) designs that correspond to
the coefficients in the complete Jacobi polynomials. 
Bonnecaze et al.~\cite{BSU2001} investigated the $3$-colored $t$-designs 
structure for the extremal Type~$\III$ codes. 
In this paper, we study the Type~$\IV$ codes of some 
specific lengths and obtain the $4$-colored $t$-design structures.

This paper is organized as follows. 
In Section~\ref{Sec:Preli}, 
we discuss the basic definitions and notations  
that we use in this paper.
We also prove the MacWilliams type identity (Theorem~\ref{Thm:SplitCompMacWilliams})
for the spilt complete weight enumerators of codes over~$\FF_{q}$.
In Section~\ref{Sec:GenColrDesign}, 
we introduced several colored designs, namely 
generalized colored $t$-designs, colored packing (resp. covering) 
designs and some of their properties. 
In Section~\ref{Sec:DesignJac}, 
we give the MacWilliams type identity (Theorem~\ref{Thm:JacobiMacWilliams})
for the spilt complete Jacobi polynomials of linear codes over~$\FF_{q}$.
We also observe (Theorem~\ref{Thm:Main1}, Theorem~\ref{Thm:Main2}) how polarization operator acts to obtain the split complete Jacobi polynomials 
attached to multiple sets of coordinate places of a code.
In Section~\ref{Sec:Examples},
we disclose some facts between a 
Type~$\III$ (resp. Type~$\IV$) code of specific lengths 
and colored designs with the help of the complete Jacobi polynomials.
We also show that the codewords of
fixed composition in the Hermitian Type~$\IV$ codes of length~$6$ 
hold $4$-colored $2$-designs (Theorem~\ref{Thm:Main3}).
Finally, we conclude the paper with some remarks in Section~\ref{Sec:Conclusion}.

All computer calculations in this paper were done with the help of
Magma~\cite{Magma}.

\section{Preliminaries}\label{Sec:Preli}

Let $\FF_{q}$ be a finite field of order~$q$, 
where $q$ is a prime power. 
Then $\FF_{q}^{n}$ denotes the vector space of dimension $n$ over $\FF_{q}$.
The elements of $\FF_{q}^{n}$ are known as \emph{vectors}.
The \emph{Hamming weight} of a vector 
$\bm{u} = (u_1,\dots, u_n)\in \FF_{q}^{n}$ 
is denoted by~$\wt(\bm{u})$ and
defined to be the number of $i$'s such that $u_{i} \neq 0$. 
The \emph{inner product} of two vectors 
$\bm{u},\bm{v} \in \FF_{q}^{n}$ 
is given by
\[
	\bm{u} \cdot \bm{v} 
	:= 
	u_{1} v_{1} + \dots + u_{n}v_{n},
\]
where $\bm{u} = (u_{1},\ldots,u_{n})$ and $\bm{v} = (v_{1},\ldots,v_{n})$.
If $q$ is an even power of an arbitrary prime~$p$,
then it is convenient to consider 
another inner product known as
the \emph{Hermitian inner product} 
which can be defined as
\[
	\bm{u} \cdot \overline{\bm{v}}
	:=
	u_{1}\overline{v_{1}} + \cdots + u_{n}\overline{v_{n}},
\]
where $\overline{v_{i}}:= {v_{i}}^{\sqrt{q}}$.
An $\FF_q$-linear code of length~$n$ is a vector subspace of $\FF_{q}^{n}$. 
The elements of an~$\FF_{q}$-linear code are called \emph{codewords}. 
The \emph{dual code} of an $\FF_{q}$-linear code~$C$ 
of length~$n$ is defined by
\[
	C^\perp 
	:= 
	\{
		\bm{v}\in \FF_{q}^{n} 
		\mid 
		\bm{u} \cdot \bm{v} = 0 
		\text{ for all } 
		\bm{u}\in C
	\}. 
\]
An $\FF_{q}$-linear code $C$ is called \emph{self-dual} if $C = C^\perp$.
Let $q$ be an even power of an arbitrary prime number.
Then an $\FF_{q}$-linear code~$C$ of length~$n$ 
is called \emph{Hermitian self-dual} if $C = C^{\perp_{H}}$,
where $C^{\perp_{H}}$ denotes the \emph{Hermitian dual code} of $C$ 
which is defined as
\[
	C^{\perp_{H}} 
	:= 
	\{
		\bm{v}\in \FF_{q}^{n} 
		\mid 
		\bm{u} \cdot \overline{\bm{v}} = 0 
		\text{ for all } 
		\bm{u}\in C
	\}.
\]
Most of the results in this paper are stated for $\FF_{q}$-linear codes 
with usual inner product but it can be re-phrased with equal validity
to the case of the codes with the Hermitian inner product.

Let $X \subseteq [n]$. 
The \emph{composition} of an element $\bm{u}\in \FF_{q}^{n}$ 
attached to~$X$ is the $q$-tuple:
\[
	\comp_{X}(\bm{u}) 
	:= 
	(n_{a,X}(\bm{u}) : a \in \FF_{q}),
\]
where $ n_{a,X}(\bm{u}) := \#\{i \in X \mid u_{i} = a\}$.
Obviously,
$\sum_{a \in \FF_{q}} n_{a,X}(\bm{u}) = |X|$.
If $X = [n]$, we prefer to write the composition of $\bm{u}\in \FF_{q}^{n}$ 
as
\[
	\comp(\bm{u}) 	
	:= 
	(n_{a}(\bm{u}) : a \in \FF_{q}),
\]
where 
$n_{a}(\bm{u})$
denotes 
the number of coordinates of $\bm{u}$ that are equal 
to~$a \in \FF_{q}$.

It is well known that 
the length~$n$ of a self-dual code over~$\FF_q$ is even and the dimension is $n/2$.
To study self-dual codes in detail, we refer the readers 
to~\cite{BMS1972, Gleason, MMS1972, NRS}. 
A self-dual code~$C$ over~$\FF_3$ of length $n\equiv 0\pmod 4$ is called \emph{Type}~$\III$ if the weight of each codeword of~$C$ is multiple of~$3$.
Again a self-dual code~$C$ over~$\FF_4$ of length $n\equiv 0 \pmod 2$ having even weight is called \emph{Type}~$\IV$. 

\begin{df}
	Let $C$ be and $\FF_q$-linear code of length $n$.
	Then the \emph{weight enumerator} of $C$ is defined as
	\[
		W_C(x,y)
		:=
		\sum_{\bm{u}\in	C}
		x^{n-\wt(\bm{u})}y^{\wt(\bm{u})}.
	\]
\end{df}

\begin{df}\label{Def:CompWeightEnum}
	Let $C$ be an $\FF_{q}$-linear code of length~$n$.
	Then the \emph{complete weight enumerator} of $C$ 
	is defined as
	\[
		\cwe_{C}(\{x_{a}\}_{a\in\FF_{q}})
		:=
		\sum_{\bm{u} \in C}
		\prod_{a \in \FF_{q}}
		x_{a}^{n_{a}(\bm{u})}.
	\]
\end{df}

\begin{rem}
	$W_{C}(x,y) = \cwe_{C}(x_{0} \leftarrow x, \{x_{a} \leftarrow y\}_{0 \neq a\in\FF_{q}})$.
\end{rem}

\begin{df}\label{Def:SplitCompWeight}
	Let $C$ be an $\FF_{q}$-linear code of length~$n$. Then 
	the \emph{split complete weight enumerator} attached to $\ell$ mutually
	disjoint subset $X_{1},\ldots,X_{\ell}$ of 
	coordinate places of the code $C$ such that 
	$$X_{1} \sqcup \cdots \sqcup X_{\ell} = [n]$$ is defined as follows:
	\[
		\scwe_{C,X_{1},\ldots,X_{\ell}}
		(\{\{x_{X_{i},a}\}_{a\in\FF_{q}}\}_{1\leq i \leq \ell})
		:=
		\sum_{\bm{u} \in C}
		\prod_{i = 1}^{\ell}
		\prod_{a \in \FF_{q}}
		x_{X_{i},a}^{n_{a,X_{i}}(\bm{u})}.
	\]
\end{df}

Note that when $\ell = 1$, the split complete weight enumerators of an $\FF_{q}$-linear code~$C$ coincide with its complete weight enumerators. 

\begin{ex}\label{Ex:SplitCompWeightEnum}
	Let $C_{4}$ be an $\FF_{3}$-linear code of length~$4$ 
	with the generator matrix:
	\[
		\begin{bmatrix}
			1 & 0 & 1 & 1 \\
			0 & 1 & 1 & 2
		\end{bmatrix}.
	\]
	The elements of $C_{4}$ are listed as follows:
	\[
			\begin{tabular}{c c c}
				$(0,0,0,0)$, & $(0,1,1,2)$, & $(0,2,2,1)$,\\
				$(1,0,1,1)$, & $(1,1,2,0)$, & $(1,2,0,2)$,\\
				$(2,0,2,2)$, & $(2,1,0,1)$, & $(2,2,1,0)$.
			\end{tabular}
	\]
	Therefore the complete weight enumerator of $C_{4}$ is
	\[
		\cwe_{C_{4}}(x_{0},x_{1},x_{2})
		=
		x_{0}^{4}
		+
		x_{0}^{1} x_{1}^{3}
		+
		x_{0}^{1} x_{2}^{3}
		+
		3
		x_{0}^{1} x_{1}^{2} x_{2}^{1}
		+
		3
		x_{0}^{1} x_{1}^{1} x_{2}^{2}.
	\]
	Let $X_{1} = \{1,2\}$ and $X_{2} = \{3,4\}$ be two disjoint sets
	such that $X_{1} \sqcup X_{2} = [4]$. Then the split complete weight
	enumerator of $C_{4}$ attached to $X_{1}$ and $X_{2}$ is
	\begin{multline*}
		\scwe_{C_{4},X_{1},X_{2}}
		(x_{X_{1},0},x_{X_{1},1},x_{X_{1},2},x_{X_{2},0},x_{X_{2},1},x_{X_{2},2})\\
		=
		x_{X_{1},0}^{2} 
		x_{X_{2},0}^{2}
		+
		x_{X_{1},0}^{1} x_{X_{1},1}^{1} 
		x_{X_{2},1}^{2}	
		+
		x_{X_{1},0}^{1} x_{X_{1},2}^{1}
		x_{X_{2},2}^{2}	\\
		+
		x_{X_{1},0}^{1}x_{X_{1},1}^{1}
		x_{X_{2},1}^{1}x_{X_{2},2}^{1}
		+
		x_{X_{1},1}^{2}
		x_{X_{2},0}^{1} x_{X_{2},2}^{1}
		+
		x_{X_{1},1}^{1} x_{X_{1},2}^{1}
		x_{X_{2},0}^{1} x_{X_{2},1}^{1}\\
		+
		x_{X_{1},0}^{1} x_{X_{1},2}^{1}
		x_{X_{2},1}^{1} x_{X_{2},2}^{1}
		+
		x_{X_{1},1}^{1} x_{X_{1},2}^{1}
		x_{X_{2},0}^{1} x_{X_{2},2}^{1}
		+
		x_{X_{1},2}^{2}
		x_{X_{2},0}^{1} x_{X_{2},1}^{1}.
	\end{multline*}
\end{ex}
 
A \emph{character} of $\FF_{q}$, where $q =  p^{f}$ for some prime number $p$, is a homomorphism from the additive group~$\FF_{q}$ 
to the multiplicative group of non-zero complex numbers. 
We review~\cite{CM2021,CMOTxxxx,MS1977} to introduce some fixed non-trivial characters over~$\FF_{q}$. 
Now let $F(x)$ be a primitive irreducible polynomial 
of degree $f$ over $\FF_{p}$ and let $\lambda$ be a root of $F(x)$.
Then any element $a \in \FF_{q}$ has a unique representation as:
\begin{equation*}\label{EquAlpha}
	a 
	=
	a_{0} + a_{1} \lambda	
	+ a_{2} \lambda^{2}
	+ \cdots +
	a_{f-1} \lambda^{f-1},
\end{equation*} 
where $a_{i} \in \FF_{p}$.
For $b \in \FF_{q}$, 
we define $\chi_{b}(a) := \zeta_{p}^{a_{0}b_{0}+\cdots+a_{f-1}b_{f-1}}$, 
where $\zeta_{p}$ is the $p$-th primitive root $e^{2{\pi}i/p}$ of unity.
When $b \neq 0$, then $\chi_{b}$ is a non-trivial character of $\FF_{q}$.
Let $\chi$ be a non-trivial character of $\FF_{q}$. 
Then for any $a \in \FF_{q}$, 
we have the following property:
\[
	\sum_{b \in \FF_{q}}
	\chi(ab) 
	:= 
	\begin{cases}
		q & \mbox{if} \quad a = 0, \\
		0 & \mbox{if} \quad a \neq 0.
	\end{cases} 
\]

\begin{lem}[\cite{MS1977}]\label{Lem:DualIden}
	Let $C$ be an $\FF_{q}$-linear code of length~$n$. 
	For $\bm{v} \in \FF_{q}^{n}$, define
	\[
		\delta_{C^{\perp}}(\bm{v}) 
		:= 
		\begin{cases}
			1 & \mbox{if } \bm{v} \in C^{\perp}, \\
			0 & \mbox{otherwise}.
		\end{cases} 
	\]
	Then we have the following identity:
	\[
		\delta_{C^\perp}(\bm{v})
		= 
		\dfrac{1}{|C|} 
		\sum_{\bm{u} \in C} 
		\chi(\bm{u} \cdot \bm{v}).
	\]
\end{lem}

Now we have the following MacWilliams type identity for the split complete weight enumerators.
The proof of the theorem is straightforward. So we leave it for the readers.

\begin{thm}\label{Thm:SplitCompMacWilliams}
	Let $C$ be an $\FF_{q}$-linear code of length $n$.
	Again let $\chi$ be a non-trivial character of $\FF_{q}$.
	Then
	\begin{align*}
		\scwe
		&_{C^{\perp},X_{1},\ldots,X_{\ell}} 
		(\{\{x_{X_{i},a}\}_{a \in \FF_{q}}\}_{1\leq i\leq \ell})\\
		& = 
		\dfrac{1}{|C|} 
		\scwe_{C,X_{1},\ldots,X_{\ell}}
		\left(
		\left\{
		\left\{
		\sum_{b \in \FF_{q}}
		\chi(a b) 
		x_{X_{i},b}
		\right\}_{a \in \FF_{q}}
		\right\}_{1\leq i \leq \ell}
		\right).
	\end{align*}
\end{thm}

\section{Generalized colored designs}\label{Sec:GenColrDesign}

Let $v$ be a positive integer, and let
$\bm{v} := (v_{1},\ldots,v_{\ell})$ such that $v = \sum_{i = 1}^{\ell} v_{i}$.
Let $\bm{X} := (X_{1},\ldots,X_{\ell})$, where $X_{i}$'s are pairwise disjoint sets
such that  $|X_{i}| = v_{i}$ for all $i$.
Again let 
\[
	\mathcal{B} 
	\subseteq  
	B_{1} \times \cdots \times B_{\ell},
\]
where $B_{i}$ is the set of blocks corresponding to~$X_{i}$ for all $i$.

Then the \emph{generalized colored incidence structure} is a triple 
$\D := (\bm{X},\mathcal{B},\mathcal{C})$, where $\mathcal{C}$
is a set of \emph{colors}, together with a function
\[
	\rho_{i} : X_{i} \times B_{i} \to \mathcal{C}
\]
for all~$i$. 
We will say that $\mathcal{B}$ has color $\rho_{i}(p,b)$ at $p$ in the $i$-th 
component.
For an element $\bm{K}:=(K_{1},\ldots,K_{\ell}) \in \mathcal{B}$, 
we define a function $n_{i} : \mathcal{C} \to \ZZ$ called the \emph{palette} on $K_{i}$
for $1\leq i \leq \ell$ that counts the number of occurrence of color $c \in \mathcal{C}$ in $K_{i}$.
The generalized colored incident structure is said to be \emph{uniform} if 
each color $c\in \mathcal{C}$ occurs $\sum_{i=1}^{\ell} n_{i}(c)$ times in every element $\bm{K} \in \mathcal{B}$.

\begin{df}
	The uniform colored incidence structure~$\D$ is called
	the \emph{generalized} \emph{colored} $t$-\emph{design} 
	if $\bm{t} := (t_{1},\ldots,t_{\ell})$ such that $t = \sum_{i = 1}^{n} t_{i}$, 
	then for each $\bm{C}:=(C_{1},\ldots,C_{\ell})$ such that $C_{i}$ is the $t_{i}$-multiset of colors (repeated choice allowed) for all~$i$, 
	there is a number $\lambda \geq 0$ such that
	for any choice $\bm{T} := (T_{1},\ldots,T_{\ell})$ with $T_{i} \in \binom{X_{i}}{t_{i}}$ for all $i$, 
	there are precisely $\lambda$ members $\bm{K} := (K_{1},\ldots,K_{n}) \in \mathcal{B}$ for which $T_{i} \subseteq K_{i}$ for all $i$ that use the $t_{i}$-multiset of colors~$C_{i}$ for the points in~$T_{i}$.
\end{df}

Let $\D = (\bm{X},\mathcal{B},\mathcal{C})$ 
be a generalized colored $t$-design with the set of colors 
$\mathcal{C} = \{1,2,\ldots,r\}$.
It is immediate from the above definition that the~$\lambda$'s
are not independent from the choices of colors.
Therefore by
$\lambda_{(j_{1}(1),\ldots,j_{1}(r)),\ldots,(j_{\ell}(1),\ldots,j_{\ell}(r))}$,
we denote the number of $\bm{K} = (K_{1},\ldots,K_{\ell}) \in \mathcal{B}$
that uses the color~$c\in \mathcal{C}$ $j_{i}(c)$ times in $K_{i}$
with $\sum_{c=1}^{r} j_{i}(c) = t_{i}$ for all $1\leq i\leq \ell$.
Obviously, the parameters of a generalized colored $t$-design 
$\D = (\bm{X},\mathcal{B},\mathcal{C})$ depend only on the number of points $v_{i}$ in $X_{i}$ for all $i$, the number of blocks $|\mathcal{B}|$ and palette
$((n_{1}(1),n_{1}(2),\ldots,),\ldots,(n_{\ell}(1),n_{\ell}(2),\ldots))$, 
therefore it is convenient to write
a generalized colored $t$-design as a
\[
	t-(\bm{v},((n_{1}(1),n_{1}(2),\ldots,),\ldots,(n_{\ell}(1),n_{\ell}(2),\ldots)),|\mathcal{B}|)
\]
design. 
Note that 
when $\bm{k} = (k)$ and $\bm{v} = (v)$, 
then the generalized colored $t$-design 
coincide with the colored~$t$-\emph{design}.
For detail discussions on 
colored $t$-designs, 
we refer the readers 
to~\cite{BRS2000, BSU2001}.

To construct the generalized colored $t$-designs 
from a linear code~$C$ of length~$n$ over $\FF_{q}$.
Let $\bm{v} = (v_{1},\ldots,v_{\ell})$ such that 
$\sum_{i = 1}^{\ell} v_{i} = n$ and
$\bm{X} = (X_{1},\ldots,X_{\ell})$ 
of pairwise disjoint sets 
$X_{i} \subseteq [n]$ with $|X_{i}| = v_{i}$.
We define the \emph{split composition} $\bm{s}:=(s_{1},\ldots,s_{\ell})$ 
such that
$s_{i} =  (s_{i1},\ldots,s_{iq})$ for all $i$ satisfying
$\sum_{j = 1}^{q} s_{ij} = v_{i}$.
For any codeword $\bm{u} \in C$, 
let $\bm{K}(\bm{u}) := (K_{X_{1}}(\bm{u}),\ldots,K_{X_{\ell}}(\bm{u}))$
such that $K_{X_{i}}(\bm{u})$ for all~$i$ are the characteristic vectors
of the supports of $\bm{u}$ with the coordinate place  $X_{i}$.
Let $C_{\bm{s}}$ be the set of codewords of $\bm{u} \in C$ such that
$\comp_{X_{i}}(\bm{u}) = s_{i}$ for all $i$.
We denote
\begin{align*}
	\mathcal{B}(C_{\bm{s}}) 
	& := 
	\{\bm{K}(\bm{u}) \mid \bm{u} \in C_{\bm{s}}\}.
\end{align*}
In general, $\mathcal{B}(C_{\bm{s}})$ is a multi-set. 
Let $\mathcal{C}$ be a set of colors. Then
we call~$C_{\bm{s}}$ is a generalized colored $t$-design if
the triple $(\bm{X},\mathcal{B}(C_{\bm{s}}),\mathcal{C})$ together with
the function $\rho_{i}(p,b)$ of $u_{p}$ is a generalized colored $t$-design.
We say the code~$C$ is \emph{generalized colorwise} $t$-\emph{homogeneous} if the set of codewords $C_{\bm{s}}$ for every given $\bm{s}$ holds a generalized colored $t$-design. The code is called \emph{generalized colorwise homogeneous}
when $t=1$.
In particular, 
for any $\FF_{q}$-linear code one can choose the function $\rho_{i}(p,b) = a$
if $u_{p} = a$ where $a \in \FF_{q}$.

A \emph{colored design} with parameters
$t$-$(v,(n(1),n(2),\ldots),(\lambda_{1}^{a_{1}}(P),\ldots,\lambda_{N}^{a_{N}}(P)))$
is a set of blocks~$\mathcal{B}$ with palette $(n(1),n(2),\ldots)$
of a set of $v$ points, called the varieties 
and a partition of the set of all $t$-tuples into~$N$ groups 
$G_{1},\ldots,G_{N}$ satisfying that
for each $t$-multiset of colors~$P$ (repeated choices allowed),
there is a number~$\lambda_{i}$ such that for every $t$-set belonging to 
$G_{i}$ (say $a_{i}$ such $t$-set),
there are exactly $\lambda_{i}$-blocks in~$\mathcal{B}$ that use the $t$-multiset of colors~$P$.

When $N=1$, it is clearly a colored $t$-design.
A \emph{colored} \emph{packing} (resp. \emph{covering}) \emph{design} 
with parameters $t$-$(v,(n(1),n(2),\ldots), \lambda(P))$ 
is a colored design with $\max(\lambda_{i}) = \lambda$ (resp.
$\min(\lambda_{i}) = \lambda$). 
The minimum (resp. maximum) number of blocks of a
covering (resp. packing) design is denoted by $C_{\lambda_{j(1)},...,j(r)}(v,(n(1),\ldots,n(r)),t)$ 
(resp. $D_{\lambda_{j(1)},...,j(r)}(v,(n(1),\ldots,n(r)),t)$).
Note that the $2$-colored packing (resp. covering) designs
is the packing (resp. covering) designs in the sense of  
Bonnecaze et al.~\cite{BMS1999}.

\section{Jacobi polynomials and polarization}\label{Sec:DesignJac}

The MacWilliams type identity for the Jacobi polynomial of 
an $\FF_{q}$-linear code with one reference vector was given in~\cite{Ozeki}. 
In this section, we give the MacWilliams type identity for
the Jacobi polynomial of an $\FF_{q}$-linear code attached to multiple sets 
of coordinate positions of the code.
Bonnecaze et al.~\cite{BRS2000} defined the Aronhold polarization operator, 
and as an application of this operator, 
they obtained a formula to evaluate the Jacobi polynomial 
of a $\ZZ_{4}$-code from the symmetrized weight enumerator of the code. 
Later an analogue of the formula was given in~\cite{BSU2001} for complete weight
enumerators of $\FF_3$-linear codes. 
In this section, we give the generalizations of the polarization operation, 
and using these operators, we evaluate the complete 
Jacobi polynomial of an $\FF_q$-linear code attached to the multiple reference sets.

\begin{df}\label{Def:JacOne}
	Let $C$ be an $\FF_{q}$-linear code of length~$n$. Then 
	the \emph{Jacobi polynomial} attached to a set $T$ of 
	coordinate places of the code~$C$ is defined as follows:
	\[
		J_{C,T}(w,z,x,y) 
		:=
		\sum_{\bm{u}\in C}
		w^{m_{0,T}(\bm{u})}
		z^{m_{1,T}(\bm{u})}
		x^{m_{0,[n]\setminus T}(\bm{u})}
		y^{m_{1,[n]\setminus T}(\bm{u})},
	\]
	where $T\subseteq [n]$, and for $\bm{u}\in C$,
	\begin{align*}
		m_{0,T}(\bm{u}) & := \#\{i\in T \mid u_i=0\},\\
		m_{1,T}(\bm{u}) & := \#\{i \in T \mid u_{i} \neq 0\},\\
		m_{0,[n]\setminus T}(\bm{u}) & := \#\{i\in[n]\setminus T \mid u_i= 0\},\\
		m_{1,[n]\setminus T}(\bm{u})  & := \#\{i\in[n]\setminus T \mid u_i \neq 0\}.
	\end{align*}
\end{df}

\begin{rem}
	If $T \subseteq [n]$ is empty, then $J_{C,T}(w,z,x,y)=W_C(x,y)$.
\end{rem}

\begin{df}\label{Def:CompleteJac}
	Let $C$ be an $\FF_{q}$-linear code of length~$n$. Then 
	the \emph{complete Jacobi polynomial} attached to a set $T$ of 
	coordinate places of the code $C$ is defined as follows:
	\[
		\CJ_{C,T}(\{x_{a},y_{a}\}_{a\in\FF_{q}}) 
		:=
		\sum_{\bm{u}\in C}
		\prod_{a \in \FF_{q}}
		x_{a}^{n_{a,T}(\bm{u})}
		y_{a}^{n_{a,[n]\backslash T}(\bm{u})},
	\]
	where $T\subseteq [n]$, and $n_{a,T}(\bm{u})$ is the composition of 
	$\bm{u}$ on $T$ and $n_{a,[n]\backslash T}(\bm{u})$ is the composition 
	of $\bm{u}$ on $[n]\backslash T$.
\end{df}

\begin{rem}
	$J_{C,T}(w,z,x,y) = \CJ_{C,T}(x_{0} \leftarrow w, \{x_{a} \leftarrow z\}_{0 \neq a\in\FF_{q}},y_{0} \leftarrow x, \{y_{a} \leftarrow y\}_{0 \neq a\in\FF_{q}})$.
\end{rem}

\begin{df}
	Let $C$ be an $\FF_{q}$-linear code of length~$n$.
	Let $X_{1},\ldots, X_{\ell}$ be $\ell$ mutually
	disjoint sets such that 
	$$[n] = X_{1} \sqcup \cdots \sqcup X_{\ell}.$$ 
	Then the \emph{split complete Jacobi polynomial} of $C$ 
	attached to $T_{1},\ldots,T_{\ell}$ such that 
	$T_{i} \subseteq X_{i}$ for all $i$
	is defined by
	\begin{align*}
		\SCJ_{C,X_{1}(T_{1}),\ldots,X_{\ell}(T_{\ell})}
		(\{\{x_{X_{i}, a}, y_{X_{i},a}\}_{a\in\FF_{q}}\}_{1\leq i \leq \ell})
		& := 
		\sum_{\bm{u}\in C} 
		\prod_{i = 1}^{\ell}
		\prod_{a \in \FF_{q}}
		x_{X_{i},a}^{n_{a,T_{i}}(\bm{u})}
		y_{X_{i},a}^{n_{a,X_{i}\setminus T_{i}}(\bm{u})}.
	\end{align*} 
\end{df}

Note that if $\ell=1$, 
the above definition is completely equivalent to the complete Jacobi polynomial with one 
reference vector (Definition~\ref{Def:JacOne}).

\begin{ex}\label{Ex:CompJacPoly}
	Let us consider the code~$C_{4}$ from Example~\ref{Ex:SplitCompWeightEnum}.
	Then the complete Jacobi polynomial of $C_{4}$ attached to a set of coordinate
	places $T=\{1,3\}$ is
	\begin{multline*}
		\CJ_{C_{4},T}
		(x_{0},x_{1},x_{2},y_{0},y_{1},y_{2})
		=
		x_{0}^{2} y_{0}^{2}
		+
		x_{1}^{2} y_{0}^{1} y_{1}^{1}
		+
		x_{2}^{2} y_{0}^{1} y_{2}^{1}
		+
		x_{0}^{1} x_{1}^{1} y_{1}^{1} y_{2}^{1}\\
		+
		x_{1}^{1} x_{2}^{1} y_{0}^{1} y_{1}^{1}
		+
		x_{0}^{1} x_{2}^{1} y_{1}^{2}
		+
		x_{0}^{1} x_{2}^{1} y_{1}^{1} y_{2}^{1}
		+
		x_{0}^{1} x_{1}^{1} y_{2}^{2}
		+
		x_{1}^{1} x_{2}^{1} y_{0}^{1} y_{2}^{1}
	\end{multline*}
	We consider the same $X_{1}$ and $X_{2}$ from Example~\ref{Ex:SplitCompWeightEnum}.
	Let $T_{1} = \{1\}$ and $T_{2} = \{3\}$. 
	Then the split complete Jacobi polynomials of $C_{4}$ attached to 
	$T_{1} \subseteq X_{1}$ and $T_{2} \subseteq X_{2}$ is as follows:
	\begin{multline*}
		\SCJ_{C_{4},X_{1}(T_{1}),X_{2}(T_{2})}
		(\{\{x_{X_{i},a},y_{X_{i},a}\}_{a\in\FF_{q}}\}_{i=1,2})\\
		=
		x_{X_{1},0}^{1} 
		y_{X_{1},0}^{1} 
		x_{X_{2},0}^{1} 
		y_{X_{2},0}^{1} 
		+
		x_{X_{1},1}^{1} 
		y_{X_{1},0}^{1}
		x_{X_{2},1}^{1}
		y_{X_{2},1}^{1}
		+
		x_{X_{1},2}^{1}
		y_{X_{1},0}^{1}
		x_{X_{2},2}^{1}
		y_{X_{2},2}^{1}\\
		+
		x_{X_{1},0}^{1} 
		y_{X_{1},1}^{1}
		x_{X_{2},1}^{1}
		y_{X_{2},2}^{2}
		+
		x_{X_{1},1}^{1}
		y_{X_{1},1}^{1}
		x_{X_{2},2}^{1}
		y_{X_{2},0}^{1}
		+
		x_{X_{1},2}^{1}
		y_{X_{1},1}^{1}
		x_{X_{2},0}^{1}
		y_{X_{2},1}^{1}\\
		+
		x_{X_{1},0}^{1}
		y_{X_{1},2}^{1}
		x_{X_{2},2}^{1}
		y_{X_{2},1}^{1}
		+
		x_{X_{1},1}^{1}
		y_{X_{1},2}^{1}
		x_{X_{2},0}^{1}
		y_{X_{2},2}^{1}
		+
		x_{X_{1},2}^{1}
		y_{X_{1},2}^{1}
		x_{X_{2},1}^{1}
		y_{X_{2},0}^{1}		
	\end{multline*}
\end{ex}

The complete Jacobi polynomial of an $\FF_{q}$-linear code 
attached to multiple sets satisfies the following MacWilliams type identity.

\begin{thm}[MacWilliams Identity]\label{Thm:JacobiMacWilliams}
	Let $C$ be an $\FF_{q}$-linear code of length $n$.
	Again let $\chi$ be a non-trivial character of $\FF_{q}$.
	Then
	\begin{align*}
		\SCJ&_{C^{\perp},X_{1}(T_{1}),\ldots,X_{\ell}(T_{\ell})} 
		(\{\{x_{X_{i},a},y_{X_{i},a}\}_{a \in \FF_{q}}\}_{1\leq i\leq \ell})\\
		& = 
		\dfrac{1}{|C|} 
		\SCJ_{C,X_{1}(T_{1}),\ldots,X_{\ell}(T_{\ell})}
		\left(
		\left\{
			\left\{
				\sum_{b \in \FF_{q}}
				\chi(a b) 
				x_{X_{i},b},
				\sum_{b \in \FF_{q}}
				\chi(a b) 
				y_{X_{i},b}
			\right\}_{a \in \FF_{q}}
		\right\}_{1\leq i \leq \ell}
		\right).
	\end{align*}
\end{thm}

\begin{proof}
	By Lemma~\ref{Lem:DualIden}, we can write
	
	\begin{align*}
		\SCJ&_{C^{\perp},X_{1}(T_{1}),\ldots,X_{\ell}(T_{\ell})} 
		(\{\{x_{X_{i},a},y_{X_{i},a}\}_{a \in \FF_{q}}\}_{1\leq i\leq \ell})\\
		& = 
		\sum_{\bm{u}\in C^{\perp}} 
		\prod_{i = 1}^{\ell}
		\prod_{a \in \FF_{q}}
		x_{X_{i},a}^{n_{a,T_{i}}(\bm{u})}
		y_{X_{i},a}^{n_{a,X_{i}\setminus T_{i}}(\bm{u})}\\
		& = 
		\sum_{\bm{v} \in \FF_{q}^{n}} 
		\delta_{C^\perp}(\bm{v}) 
		\prod_{i = 1}^{\ell}
		\prod_{a \in \FF_{q}}
		x_{X_{i},a}^{n_{a,T_{i}}(\bm{v})}
		y_{X_{i},a}^{n_{a,X_{i}\setminus T_{i}}(\bm{v})}\\
		& =
		\dfrac{1}{|C|} 
		\sum_{\substack{\bm{u} \in C\\ 
			\bm{v}\in \FF_{q}^{n}}}  
		\chi(\bm{u} \cdot \bm{v}) 
		\prod_{i = 1}^{\ell}
		\prod_{a \in \FF_{q}}
		x_{X_{i},a}^{n_{a,T_{i}}(\bm{v})}
		y_{X_{i},a}^{n_{a,X_{i}\setminus T_{i}}(\bm{v})}\\
		& = 
		\dfrac{1}{|C|}
		\sum_{\substack{\bm{u} \in C\\ 
				\bm{v} \in \FF_{q}^{n}}} 
		\chi(u_{1}v_{1} + \cdots + u_{n}v_{n})
		\prod_{i = 1}^{\ell}
		\left(
		\prod_{j \in T_{i}}
		x_{X_{i},v_{j}}
		\right)
		\left(
		\prod_{j \in X_{i}\setminus T_{i}}
		y_{X_{i},v_{j}}
		\right)\\
		& = 
		\dfrac{1}{|C|}
		\sum_{\bm{u} \in C}
		\prod_{i = 1}^{\ell}
		\left(
		\prod_{j \in T_{i}} 
		\sum_{v_{j} \in \FF_{q}} 
		\chi(u_{j}v_{j}) 
		x_{X_{i},v_{j}} 
		\right)
		\left(
		\prod_{j \in X_{i}\setminus T_{i}} 
		\sum_{v_{j} \in \FF_{q}} 
		\chi(u_{j}v_{j}) 
		y_{X_{i},v_{j}} 
		\right)\\
		& = 
		\dfrac{1}{|C|}
		\sum_{\bm{u} \in C}
		\prod_{i=1}^{\ell}
		\prod_{a\in\FF_{q}} 
		\left(
			\sum_{b \in \FF_{q}}
			\chi(ab) 
			x_{X_{i},b}
		\right)
		^{n_{a,T_{i}}(\bm{u})}
		\left(
		\sum_{b \in \FF_{q}}
		\chi(ab) 
		y_{X_{i},b}
		\right)
		^{n_{a,X_{i}\setminus T_{i}}(\bm{u})}  \\
		& = 
		\dfrac{1}{|C|}
		\SCJ_{C,X_{1}(T_{1}),\ldots,X_{\ell}(T_{\ell})}
		\left(
		\left\{
		\left\{
		\sum_{b \in \FF_{q}}
		\chi(a b) 
		x_{X_{i},b},
		\sum_{b \in \FF_{q}}
		\chi(a b) 
		y_{X_{i},b}
		\right\}_{a \in \FF_{q}}
		\right\}_{1\leq i \leq \ell}
		\right).	
	\end{align*}
	Hence the proof is completed.
\end{proof}

The following result reflects the basic motivation to introduce 
the concept of split complete Jacobi polynomials attached to multiple sets. 
We omit the proof of the theorem since it follows from the definitions. 

\begin{thm}\label{Thm:JacToDesign}
	Let $C$ be a linear code of length~$n$ over $\FF_{q}$.
	Let $\bm{v} := (v_{1},\ldots,v_{\ell})$ such that $\sum_{i}^{\ell}v_{i} = n$.
	Let $\bm{X} := (X_{1},\ldots,X_{\ell})$ of pairwise disjoint set 
	$X_{i}\subseteq [n]$ 
	with $|X_{i}| = v_{i}$ for all $i$.
	Then the set of codewords of $C$ for every given 
	split composition forms a generalized $q$-colored $t$-design
	with $\bm{t} = (t_{1},\ldots,t_{\ell})$ such that 
	$\sum_{i = 1}^{\ell}t_{i} = t$
	if and only if
	the split complete Jacobi polynomial 
	$\SCJ_{C,X_{1}(T_{1}),\ldots,X_{\ell}(T_{\ell})}$ 
	with $T_{i} \in \binom{X_{i}}{t_{i}}$
	for all $i$ is independent of the choices of the sets $T_{1},\ldots,T_{\ell}$.
\end{thm}


Let $C$ be a linear code of length~$n$ over $\FF_{q}$.
Then the code $C-i$ (resp. $C/i$) obtained from~$C$ by \emph{puncturing} (resp. \emph{shortening}) at coordinate place $i$. 
We denote by $C+i_{a}$ for all $a \in \FF_{q}$ the subcodes of $C$ where the
$i$-th entry of each codeword takes the value $a$ punctured at $i$.

Let $\ell$, $n$ be the positive integers. 
Let $\bm{v}:= (v_{1},\ldots,v_{\ell})$ 
such that $\sum_{i=1}^{\ell} v_{i} = n$. 
Let $P(\{\{x_{X_{i},a},y_{X_{i},a}\}_{a \in \FF_{q}}\}_{1\leq i \leq \ell})$ be a polynomial of degree~$n$ in $2q\ell$ variables 
such that in its each term the sum of the powers of 
$x_{X_{i},a}$ and $y_{X_{i},a}$ for all 
$a\in\FF_{q}$ is~$v_{i}$ for all~$i$. 
Again let $P'_{k,a}$ denote the partial derivative with respect to 
$y_{X_{k},a}$
for any integer $1\leq k \leq \ell$ and $a\in\FF_{q}$. 
Define the polarization operator 
$A_{2m\ell,k}$ 
for any integer $1 \leq k \leq \ell$
as follows:
\begin{equation*}\label{Equ:Operator}
	A_{2q\ell,k}
	\cdot
	 P
	:= 
	\frac{1}{v_{k}}
	\sum_{a\in\FF_{q}}
	x_{X_{k},a} 
	P_{k,a}'.
\end{equation*}
For $\ell = 1$, 
it convenient to denote the polarization operator as $A_{2q}$
instead of $A_{2q\ell,k}$. The detail of this particular case
is as follows:

Let $P(\{x_{a},y_{a}\}_{a\in\FF_{q}})$ be a polynomial of degree~$n$
in $2q$ variables.
Let~$P_{y_{a}}'$ be the partial derivative with respect to $y_{a}$
for $a\in\FF_{q}$. Then
\[
	A_{2q}
	\cdot
	P
	:= 
	\frac{1}{n}
	\sum_{a \in \FF_{q}}
	x_{a} 
	P_{y_{a}}'.	
\]

Now we have the following generalization of the $\FF_{q}$-analogue of~\cite[Theorem 1]{BRS2000}.

\begin{thm}\label{Thm:Main1}
	Let $C$ be a linear code of length $n$ over $\FF_{q}$. 
	Let $\bm{v} := (v_{1},\ldots,v_{\ell})$ 
	such that $\sum_{i=1}^{\ell} v_{i} = n$.
	We also let $X_{1},\ldots,X_{\ell}$
	be the mutually disjoint subsets of $[n]$ such that $X_{1} \sqcup\cdots\sqcup X_{\ell} = [n]$ and $|X_{i}| = v_{i}$ for all $i$. 
	Then for every coordinate place $i \in X_{k}$ for $1\leq k \leq \ell$, we have
	\begin{multline*}
		\SCJ_{C,X_{1}(\emptyset),\ldots,X_{k-1}(\emptyset),X_{k}(\{i\}),X_{k+1}(\emptyset),\ldots,X_{\ell}(\emptyset)}\\
		=
		x_{X_{k},0} 
		\scwe_{C/i,X_{1},\ldots,X_{k-1},X_{k},X_{k+1},\ldots,X_{\ell}} \\
		+ 
		\sum_{a \in \FF_{q}, a\neq 0}
		x_{X_{k},a}
		\scwe_{C+i_{a},X_{1},\ldots,X_{k-1},X_{k},X_{k+1},\ldots,X_{\ell}}.
	\end{multline*}
	If $C$ contains no nonzero codewords of Hamming weight less than~$1$,
	we have
	\begin{multline*}
		v_{i}
		(A_{2q\ell,k} 
		\cdot
		\scwe_{C,X_{1},\ldots,X_{k-1},X_{k},X_{k+1},\ldots,X_{\ell}})\\
		=
		x_{X_{k},0} 
		\sum_{i\in X_{k}}
		\scwe_{C/i,X_{1},\ldots,X_{k-1},X_{k},X_{k+1},\ldots,X_{\ell}}\\
		+
		\sum_{a \in \FF_{q}, a\neq 0}
		x_{X_{k},a}
		\sum_{i\in X_{k}}
		\scwe_{C+i_{a},X_{1},\ldots,X_{k-1},X_{k},X_{k+1},\ldots,X_{\ell}}.
	\end{multline*}
	If $C$ is a generalized colorwise homogeneous code, 
	then 
	\[
		\SCJ_{C,X_{1}(\emptyset),\ldots,X_{k-1}(\emptyset),X_{k}(\{i\}),X_{k+1}(\emptyset),\ldots,X_{\ell}(\emptyset)}
		=
		A_{{2q\ell,k}}
		\cdot
		\scwe_{C,X_{1},\ldots,X_{k-1},X_{k},X_{k+1},\ldots,X_{\ell}}.
	\]
\end{thm}

\begin{proof}
	The proof follows the similar arguments given in~\cite[Theorem 1]{BRS2000}.
	So we omit the details.
\end{proof}

The following theorem is an analogue of the above theorem for $t > 1$. 
We leave the proof for the readers.

\begin{thm}\label{Thm:Main2}
	Let $C$ be a linear code of length $n$ over $\FF_{q}$. 
	Let $\bm{v} := (v_{1},\ldots,v_{\ell})$ 
	such that $\sum_{i=1}^{\ell} v_{i} = n$.
	We also let $X_{1},\ldots,X_{\ell}$
	be the mutually disjoint subsets of $[n]$ such that $X_{1} \sqcup\cdots\sqcup X_{\ell} = [n]$ and $|X_{k}| = v_{k}$ for all $k$.
	If $C$ is a generalized colorwise $t$-homogeneous and contains no codeword of
	Hamming weight less than $t$, 
	then for $\bm{t} := (t_{1},\ldots,t_{\ell})$ 
	such that $\sum_{i=1}^{\ell}t_{i} = t$, 
	we have
	\[
		\SCJ_{X_{1}(T_1),\ldots,X_{\ell}(T_{\ell})}
		=
		A_{2q\ell,\ell}^{t_{\ell}}\cdots A_{2q\ell,1}^{t_{1}}
		\cdot
		\scwe_{C,X_{1},\ldots,X_{\ell}},
	\] 
	for each $(T_{1},\ldots,T_{\ell}) \in \binom{X_{1}}{t_{1}} \times \cdots \times \binom{X_{\ell}}{t_{\ell}}$.
\end{thm}

\section{Application to colored design theory}\label{Sec:Examples}

Bonnecaze et al.~\cite{BSU2001} 
gave colored $t$-design structures using Type~III codes. 
In this section, using their idea, we construct some (generalized) colored
designs such as colored packing (resp. covering) designs 
using Type~III and Type~IV codes.

\subsection{Invariant theory}

Let $G$ be a finite $n\times n$ matrix group that acts 
on a polynomial ring $\CC[x_0,\ldots,x_{n-1}]$; 
for $g\in G$ and $f(x_0,\ldots,x_{n-1})\in \CC[x_0,\ldots,x_{n-1}]$, 
\[
	gf(x_0,\ldots,x_{n-1})=f(g(x_0,\ldots,x_{n-1})^t). 
\]
Then 
\[
	\widehat{G}
	=
	\left\{
	\left.
	\begin{pmatrix}
		g&0\\
		0&g
	\end{pmatrix}
	\right\vert 
	g\in G
	\right\}
\]
acts on a polynomial ring 
$\CC[x_0,\ldots,x_{n-1},y_0,\ldots,y_{n-1}]$ in a natural way. 
Let $M^{\widehat{G}}_{i,j}
=\CC[x_0,\ldots,x_{n-1},y_0,\ldots,y_{n-1}]^{\widehat{G}}_{i,j}$ 
be the invariants of degree $(i,j)$: 
\begin{align*}
	&\CC[x_0,\ldots,x_{n-1},y_0,\ldots,y_{n-1}]^{\widehat{G}}_{i,j}\\
	&=\{f\in \CC[x_0,\ldots,x_{n-1},y_0,\ldots,y_{n-1}]\mid 
	(g,h)f=f, \mbox{degree of $f$ in $\{x_k\}$ is $i$}, \\
	&\hspace{30pt}\mbox{degree of $f$ in $\{y_k\}$ is $j$}\}. 
\end{align*}
In \cite{Stanley1979}, Stanley defined
the bivariate Molien series 
\[
	f(u,v)
	=\sum_{u,v}\dim (M^{\widehat{G}}_{i,j})u^iv^j,
\]
and showed that $f(u,v)$ is written as follows: 
\[
	f(u,v)
	=
	\frac{1}{|G|}
	\sum_{g\in G}
	\frac{1}{\det{(1-ug)}\det{(1-vg)}}.
\]

We denote the homogeneous part of degree $d$ of $f(u,v)$ by $f[d]$.
To obtain an invariant, the Reynolds operator is useful. 
For $f\in \CC[x_0,\ldots,x_{n-1},y_0,\ldots,y_{n-1}]$, 
the Reynolds operator of $f$ and $\widehat{G}$ is defined as follows: 
\[
	R(f,\widehat{G})
	:=
	\sum_{(g,g)\in \widehat{G}}
	(g,g)\cdot f. 
\]
Then it is easy to show that $R(f,\widehat{G})$ is an invariant of $\widehat{G}$.

\subsection{Type III codes}

The MacWilliams transform and some congruence conditions yield that 
the complete weight enumerator of a Type~III code remains invariant under the action of 
group $G_{\rm{III}}$ of order $2592$ which is generated by the following four matrices:
\begin{align*}
	\frac{1}{\sqrt{3}}
	&\begin{bmatrix}
		1&1&1\\
		1&e^{2\pi i/3}&e^{4\pi i/3}\\
		1&e^{4\pi i/3}&e^{2\pi i/3}
	\end{bmatrix},\;
	\begin{bmatrix}
		1&0&0\\
		0&e^{2\pi i/3}&0\\
		0&0&1
	\end{bmatrix},\;\\
	&\begin{bmatrix}
		1&0&0\\
		0&1&0\\
		0&0&e^{2\pi i/3}
	\end{bmatrix},\;
	\begin{bmatrix}
		e^{\pi i/6}&0&0\\
		0&e^{\pi i/6}&0\\
		0&0&e^{\pi i/6}
	\end{bmatrix}.
\end{align*}
It is easy to show that 
a split complete weight enumerator of a Type III code 
is an invariant of $\widehat{G_{\rm{III}}}$. 

Now we assume that $G$ is a $2\times 2$ matrix group and 
$\widehat{G}$ acts on $\CC[x_0,x_1,y_0,y_1]$. 
Let $P\in \CC[x_0,x_1,y_0,y_1]$ be a polynomial of total degree~$n$. 
Now define a polarization operator 
$A_4$ as follows:
\[
A_4 \cdot P
:=
\frac{x_0P'_{y_0}+x_1P'_{y_1}}{n}.
\]

\begin{df}
	A linear code of length~$n$ over $\FF_{3}$ 
	is said to be $t$-homogeneous if the codewords of 
	every given Hamming weight hold a $t$-design.
\end{df}

\begin{df}
	An $\FF_{3}$-linear code of length~$n$ is said to be 
	\emph{colorwise $t$-homogeneous} if the codewords of every 
	given composition hold a $3$-colored $t$-design.
\end{df}

\begin{lem}[\cite{BSU2001}]\label{Lem:BSULem3}
	Let $C$ be a linear code of length~$n$ over $\FF_{3}$. 
	If~$C$ is $t$-homogeneous 
	with no non-zero words of Hamming weight less than~$t$,
	then for all $T$ of size $t$ we get
	\[
		J_{C,T} = A_4^t \cdot W_C.
	\]
\end{lem}

Let $P$ be a polynomial of total degree $n$ 
in $6$ variables $x_0,x_1,x_2,y_0,y_1,y_2$.
Now define a polarization operator 
$A_6$  and specialization operator~$S$ as follows:
\begin{align*}
	A_6 \cdot P
	&:=
	\frac{x_0P'_{y_0}+x_1P'_{y_1}+x_2P'_{y_2}}{n},\\
	S_{6} \cdot P(x_0,x_1,x_2,y_0,y_1,y_2)
	&:=
	P(x_0,x_1,x_1,y_0,y_1,y_1).
\end{align*}

\begin{lem}[\cite{BSU2001}]\label{Lem:BSULem4}
	Let $C$ be a linear code of length~$n$ over $\FF_{3}$.
	If~$C$ is colorwise $t$-homogeneous 
	with no non-zero words of Hamming weight less than~$t$,
	then for all $T$ of size $t$ we get
	\[
	\CJ_{C,T}
	=
	A_6^t \cdot \cwe_{C}.
	\]
\end{lem}

\subsubsection{Length $12$}

\begin{ex}[length $12$]
	Let $C_{12}^{\text{III}}$ be the first ternary self-dual code of length $12$ in \cite{HM}. 
	\begin{align*}
		f[12]&=2u^{12}+2u^{11}v+3u^{10}v^2+4u^3v^9+\cdots.
	\end{align*}
	
	In this case, it holds the following lemmas.
	
	\begin{lem}[\cite{BSU2001}] \label{lem:M12}
		A basis of $M^{\widehat{G_{\rm{III}}}}_{3,9}$ is obtained by applying $R(f,\widehat{G_{\rm{III}}})$ with~$f$ running over the monomials
		\[
		x_0^3y_0^9,\;x_0^3y_0^3y_1^6,\;x_0^3y_0^3y_1^3y_2^2,\;x_0^2x_1y_0^4y_1^5.
		\]
	\end{lem}
	
	
	\begin{lem}[\cite{BSU2001}] \label{lem:S12}
		For $\ell=1,2,3$ we have
		\[
			\dim (S_{6}\cdot M_{\ell,12-\ell}^{\widehat{G_{\rm{III}}}})
			=
			\dim(M_{\ell,12-\ell}^{\widehat{G_{\rm{III}}}}).
		\]
	\end{lem}
	
	Combining the preceding lemmas and the bivariate Molien series, 
	we obtain the following Theorem.
	
	\begin{thm}[\cite{BSU2001}]
		The codewords of fixed composition in the ternary Golay hold $3$-colored $3$-designs.
	\end{thm}
	Since the codewords of fixed composition in $C_{12}^{\text{III}}$ holds $3$-colored $3$-design, we assume that
	$|T|=1,2,3$.
	Then
	\begin{align*}
		\CJ_{C_{12}^{\text{III}},1}&=x_0(y_0^{11}+11y_0^5y_1^6+110y_0^5y_1^3y_2^3+11y_0^5y_2^6+55y_0^2y_1^6y_2^3+55y_0^2y_1^3y_2^6)\\
		&~+x_1(y_1^{11}+11y_0^6y_1^5+55y_0^6y_1^2y_2^3+110y_0^3y_1^5y_2^3+55y_0^3y_1^2y_2^6+11y_1^5y_2^6)\\
		&~+x_2(y_2^{11}+55y_0^6y_1^3y_2^2+11y_0^6y_2^5+55y_0^3y_1^6y_2^2+110y_0^3y_1^3y_2^5+11y_1^6y_2^5),\\
		\CJ_{C_{12}^{\text{III}},2}&=x_0^2(y_0^{10}+5y_0^4y_1^6+50y_0^4y_1^3y_2^3+5y_0^4y_2^6+10y_0y_1^6y_2^3+10y_0y_1^3y_2^6)\\
		&~+x_0x_1(12y_0^5y_1^5+60y_0^5y_1^2y_2^3+60y_0^2y_1^5y_2^3+30y_0^2y_1^2y_2^6)\\
		&~+x_0x_2(60y_0^5y_1^3y_2^2+12y_0^5y_2^5+30y_0^2y_1^6y_2^2+60y_0^2y_1^3y_2^5)\\
		&~+x_1^2(5y_0^6y_1^4+10y_0^6y_1y_2^3+50y_0^3y_1^4y_2^3+10y_0^3y_1y_2^6+5y_1^4y_2^6+y_1^{10})\\
		&~+x_1x_2(30y_0^6y_1^2y_2^2+60y_0^3y_1^5y_2^2+60y_0^3y_1^2y_2^5+12y_1^5y_2^5)+x_2^2(10y_0^6y_1^3y_2\\
		&~+5y_0^6y_2^4+10y_0^3y_1^6y_2+50y_0^3y_1^3y_2^4+5y_1^6y_2^4+y_2^{10}),\\
		\CJ_{C_{12}^{\text{III}},3}&=x_0^3(y_0^9+2y_0^3y_1^6+20y_0^3y_1^3y_2^3+2y_0^3y_2^6+y_1^6y_2^3+y_1^3y_2^6)+x_0^2x_1(9y_0^4y_1^5\\
		&~+45y_0^4y_1^2y_2^3+18y_0y_1^5y_2^3+9y_0y_1^2y_2^6)+x_0^2x_2(45y_0^4y_1^3y_0^2+9y_0^4y_2^5\\
		&~+9y_0y_1^6y_2^2+18y_0y_1^3y_2^5)+x_0x_1^2(9y_0^5y_1^4+18y_0^5y_1y_2^3+45y_0^2y_1^4y_2^3\\
		&~+9y_0^2y_1y_2^6)+x_0x_2^2(18y_0^5y_1^3y_2+9y_0^5y_2^4+9y_0^2y_1^6y_2+45y_0^2y_1^3y_2^4)\\
		&~+x_0x_1x_2(54y_0^5y_1^2y_2^2+54y_0^2y_1^5y_2^2+54y_0^2y_1^2y_2^5)+x_1^3(2y_0^6y_1^3+y_1^9\\
		&~+y_0^6y_2^3+20y_0^3y_1^3y_2^3+y_0^3y_2^6+2y_1^3y_2^6)+x_1^2x_2(9y_0^6y_1y_2^2+45y_0^3y_1^4y_2^2\\
		&~+18y_0^3y_1y_2^5+9y_1^4y_2^5)+x_1x_2^2(9y_0^5y_1y_2^3+18y_0^3y_1^5y_2+45y_0^3y_1^2y_2^4\\
		&~+9y_1^5y_2^4)+x_2^3(y_0^6y_1^3+2y_0^6y_2^3+y_0^3y_1^6+y_2^9+20y_0^3y_1^3y_2^3+2y_1^6y_2^3).
	\end{align*}
	
	\begin{cor}[\cite{BSU2001}]
		There exist simple $3$-colored $3$-designs with the following parameters:
		Three designs with parameters $3$-$(12,(n(0),n(1),n(2)),220)$ where $(n(0),n(1),n(2))$ is equal to
		$(6,3,3)$ or any one of its three permutations.
		Three designs with parameters $3$-$(12,(n(0),n(1),n(2)),22)$ where $(n(0),n(1),n(2))$ is equal to
		$(6,6,0)$ or any one of its three permutations.	
	\end{cor}
	
	The space of Jacobi polynomials $\CJ_{C_{12}^{\text{III}},T}$ with $|T|=4$ may be generated by the two polynomials
	\begin{align*}
		C^1_{C_{12}^{\text{III}},4}&=x_0^4(y_0^8+y_0^2y_1^6+6y_0^2y_1^3y_2^3+y_0^2y_2^6)+x_0^3x_1(4y_0^3y_1^5+28y_0^3y_1^2y_2^3\\
		&~+4y_1^5y_2^3)+x_0^3x_2(28y_0^3y_1^3y_2^2+4y_0^3y_2^5+4y_1^3y_2^5)+x_0^2x_1^2(12y_0^4y_1^4\\
		&~+18y_0^4y_1y_2^3+18y_0y_1^4y_2^3+6y_0y_1y_2^6)+x_0^2x_1x_2(60y_0^4y_1^2y_2^2+24y_0y_1^5y_2^2\\
		&~+24y_0y_1^2y_2^5)+x_0^2x_2^2(18y_0^4y_1^3y_2+12y_0^4y_2^4+6y_0y_1^6y_2+18y_0y_1^3y_2^4)\\
		&~+x_0x_1^3(4y_0^5y_1^3+4y_0^5y_2^3+28y_0^2y_1^3y_2^3)+x_0x_1^2x_2(24y_0^5y_1y_2^2+60y_0^2y_1^4y_2^2\\
		&~+24y_0^2y_1y_2^5)+x_0x_1x_2^2(24y_0^5y_1^2y_2+24y_0^2y_1^5y_2+60y_0^2y_1^2y_2^4)\\
		&~+x_0x_2^3(4y_0^5y_1^3+4y_0^5y_2^3+28y_0^2y_1^3y_2^3)+x_1^4(y_0^6y_1^2+6y_0^3y_1^2y_2^3+y_1^8\\
		&~+y_1^2y_2^6)+x_1^3x_2(28y_0^3y_1^3y_2^2+4y_0^3y_2^5+4y_1^3y_2^5)+x_1^2x_2^2(6y_0^6y_1y_2\\
		&~+18y_0^3y_1^4y_2+18y_0^3y_1y_2^4+12y_1^4y_2^4)+x_1x_2^3(4y_0^3y_1^5+28y_0^3y_1^2y_2^3\\
		&~+4y_1^5y_2^3)+x_2^4(y_0^6y_2^2+6y_0^3y_1^3y_2^2+y_1^6y_2^2+y_2^8),\\
		C^2_{C_{12}^{\text{III}},4}&=x_0^4(y_0^8+8y_0^2y_1^3y_2^3)+x_0^3x_1(8y_0^3y_1^5+24y_0^3y_1^2y_2^3+4y_1^2y_2^6)\\
		&~+x_0^3x_2(24y_0^3y_1^3y_2^2+8y_0^3y_2^5+4y_1^6y_2^2)+x_0^2x_1^2(6y_0^4y_1^4+24y_0^4y_1y_2^3\\
		&~+24y_0y_1^4y_2^3)+x_0^2x_1x_2(60y_0^4y_1^2y_2^2+24y_0y_1^5y_2^2+24y_0y_1^2y_2^5)\\
		&~+x_0^2x_2^2(6y_0^4y_2^4+24y_0^4y_1^3y_2+24y_0y_1^3y_2^4)+x_0x_1^3(8y_0^5y_1^3+24y_0^5y_2^3\\
		&~+4y_0^2y_2^6)+x_0x_1^2x_2(24y_0^5y_1y_2^2+60y_0^2y_1^4y_2^2+24y_0^2y_1y_2^5)\\
		&~+x_0x_1x_2^2(24y_0^5y_1^2y_2+24y_0^2y_1^5y_2+60y_0^2y_1^2y_2^4)\\
		&~+x_0x_2^3(8y_0^5y_2^3+4y_0^2y_1^6+24y_0^2y_1^3y_2^3)+x_1^4(8y_0^3y_1^2y_2^3+y_1^8)\\
		&~+x_1^3x_2(4y_0^6y_2^2+24y_0^3y_1^3y_2^2+8y_1^3y_2^5)+x_1^2x_2^2(24y_0^3y_1^4y_2+24y_0^3y_1y_2^4\\
		&~+6y_1^4y_2^4)+x_1x_2^3(4y_0^6y_1^2+24y_0^3y_1^2y_2^3+8y_1^5y_2^3)+x_2^4(8y_0^3y_1^3y_2^2+y_2^8).
	\end{align*}
	Combining these two equations we obtain $4$-designs with parameters
	\[
		4\text{-}(12,(6,6,0),(\lambda_{1}^{1}(P),\lambda_{2}^{1}(P)))
		\text{ and }
		4\text{-}(12,(6,3,3),(\lambda_{1}^{2}(P),\lambda_{2}^{2}(P)),
	\]
	where $\lambda(P)$'s are shown in Table~\ref{Tab:Type_III_4_design}.
	By the coefficient of the term $y_0^{n(0)}y_1^{n(1)}y_2^{n(2)}$ in the complete weight enumerator of the code, we obtain an upper (resp. lower) bound of $D_{\lambda_{\max}(P)}(12,(n(0),n(1),n(2)),4)$ (resp. $C_{\lambda_{\max}(P)}(12,(n(0),n(1),n(2)),4)$).
	\begin{align*}
		D_{\lambda_{\max}^{1}(P)}(12,(6,6,0),4)
		&\leq 22 \leq 
		C_{\lambda_{\min}^{1}(P)}(12,(6,6,0),4),\\
		D_{\lambda_{\max}^{2}(P)}(12,(6,3,3),4)
		&\leq 220 \leq 
		C_{\lambda_{\min}^{2}(P)}(12,(6,3,3),4).
	\end{align*}
	The $\lambda_{\max}(P)$'s (resp. $\lambda_{\min}(P)$'s) are 
	shown in Table~\ref{Tab:Type_III_4_design}.
	\begin{table}[h]
		\centering
		\caption{$\lambda$'s in $3$-colored $4$-designs}
		\label{Tab:Type_III_4_design}
		\begin{tabular}{|c||c|c|c|c|c|c|c|c|}
			\hline
			$P$ & $0000$ & $0001$ & $0002$ & $0011$ & $0012$ & $0022$ & $0111$ & $0112$ \\
			\hline
			$\lambda_{1}^{1}(P)$ & 1 & 4 & 0 & 12 & 0 & 0 & 4 & 0\\
			\hline
			$\lambda_{2}^{1}(P)$ & 0 & 8 & 0 & 6 & 0 & 0 & 8 & 0\\
			\hline
			$\lambda_{\max}^{1}(P)$ & 1 & 8 & 0 & 12 & 0 & 0 & 8 & 0\\
			\hline
			$\lambda_{\min}^{1}(P)$ & 0 & 4 & 0 & 6 & 0 & 0 & 4 & 0\\
			\hline
			$\lambda_{1}^{2}(P)$ & 6 & 28 & 28 & 18 & 60 & 18 & 4 & 24\\
			\hline
			$\lambda_{2}^{2}(P)$ & 8 & 24 & 24 & 24 & 60 & 0 & 24 & 24\\
			\hline
			$\lambda_{\max}^{2}(P)$ & 8 & 28 & 28 & 24 & 60 & 18 & 24 & 24\\
			\hline
			$\lambda_{\min}^{2}(P)$ & 6 & 24 & 24 & 18 & 60 & 0 & 4 & 24\\
			\hline
			$P$ & 0122 & 0222 & 1111 & 1112 & 1122 & 1222 & 2222 & - \\
			\hline
			$\lambda_{1}^{1}(P)$ & 0 & 0 & 1 & 0 & 0 & 0 & 0 & -\\
			\hline
			$\lambda_{2}^{1}(P)$ & 0 & 0 & 0 & 0 & 0 & 0 & 0 & -\\
			\hline
			$\lambda_{\max}^{1}(P)$ & 0 & 0 & 1 & 0 & 0 & 0 & 0 & -\\
			\hline
			$\lambda_{\min}^{1}(P)$ & 0 & 0 & 0 & 0 & 0 & 0 & 0 & -\\
			\hline
			$\lambda_{1}^{2}(P)$ & 24 & 4 & 0 & 0 & 6 & 0 & 0 & -\\
			\hline
			$\lambda_{2}^{2}(P)$ & 24 & 0 & 0 & 4 & 0 & 4 & 0 & -\\
			\hline
			$\lambda_{\max}^{2}(P)$ & 24 & 4 & 0 & 4 & 6 & 4 & 0 & -\\
			\hline
			$\lambda_{\min}^{2}(P)$ & 24 & 0 & 0 & 0 & 0 & 0 & 0 & -\\
			\hline
		\end{tabular}
	\end{table}

	Similarly, we can obtain an upper (resp. lower) bound of $D$ (resp. $C$) 
	$3$-colored $4$-designs where $(n(0),n(1),n(2))$ is equal to anyone of
	the four choices: 
	$(6,0,6)$, $(0,6,6)$, $(3,3,6)$, $(3,6,3)$.
\end{ex}

\subsection{Type IV codes}
The MacWilliams transform and some congruence conditions yield that 
the complete weight enumerator of a Type~IV code remains invariant under the action of 
group $G_{\text{IV}}$ of order $576$ which is generated by the following four matrices:
\begin{align*}
	\frac{1}{2}
	&\begin{bmatrix}
		1&1&1&1\\
		1&1&-1&-1\\
		1&-1&1&-1\\
		1&-1&-1&1
	\end{bmatrix},\;
	\begin{bmatrix}
		1&0&0&0\\
		0&-1&0&0\\
		0&0&-1&0\\
		0&0&0&-1
	\end{bmatrix},\;\\
	&\begin{bmatrix}
		0&1&0&0\\
		1&0&0&0\\
		0&0&0&1\\
		0&0&1&0
	\end{bmatrix},\;
	\begin{bmatrix}
		1&0&0&0\\
		0&0&1&0\\
		0&0&0&1\\
		0&1&0&0
	\end{bmatrix}.
\end{align*}
It is easy to show that 
a split complete weight enumerator of a Type IV code 
is an invariant of $\widehat{G_{\rm{IV}}}$. 

Let the elements of $\FF_4$ be $0,1,s,s^2$.
If $P$ is a polynomial of total degree~$n$ in $8$ variables 
$x_0,x_1,x_s,x_{s^2},y_0,y_1,y_s,y_{s^2}$.
Now define a polarization operator $A_{8}$ and specialization operator 
$\widetilde{S}$ as follows:
\begin{align*}
	A_8 \cdot P
	&:=
	\frac{x_0P'_{y_0}+x_1P'_{y_1}+x_sP'_{y_s}+x_{s^2}P'_{y_{s^2}}}{n},\\
	{S_{8}} \cdot P
	(x_0,x_1,x_s,x_{s^2},y_0,y_1,y_s,y_{s^2})
	&:=
	P(x_0,x_1,x_1,x_1,y_0,y_1,y_1,y_1).
\end{align*}

\begin{df}
	An $\FF_{4}$-linear code of length~$n$ is said to be 
	\emph{colorwise $t$-homogeneous} if the codewords of 
	every given composition hold a $4$-colored $t$-design.
\end{df}

Now we have the following $\FF_{4}$-code analogues of 
Lemma~\ref{Lem:BSULem3} and Lemma~\ref{Lem:BSULem4}.

\begin{lem}\label{Lem:BSULem3IV}
	Let $C$ be a linear code of length~$n$ over $\FF_{4}$.
	If~$C$ is $t$-homogeneous 
	with no non-zero words of Hamming weight less than~$t$,
	then for all $T$ of size $t$ we get
	\[
		J_{C,T}=A_4^t \cdot W_C.
	\]
\end{lem}

\begin{lem}\label{Lem:BSULem4IV}
	Let $C$ be a linear code of length~$n$ over $\FF_{4}$.
	If~$C$ is colorwise $t$-homogeneous 
	with no non-zero words of Hamming weight less than~$t$,
	then for all $T$ of size $t$ we get
	\[
	\CJ_{C,T}=A_8^t \cdot \cwe_{C}.
	\]
\end{lem}

\subsubsection{Length $4$}

\begin{ex}
	Let $C_{4}^{\text{IV}}$ be the Hermitian self-dual code over $\FF_4$ of length $4$ in \cite{HM}. 
	\begin{align*}
		f[4]&=u^4+u^3v+2u^2+uv^3+v^2.
	\end{align*}
	Since the codewords of fixed composition in $C_{4}^{\text{IV}}$ holds $4$-colored $1$-design, we assume that
	$|T|=1$.
	Then,
	\begin{align*}
		\CJ_{C_{4}^{\text{IV}},1}&=x_0(y_0^3+y_0y_1^2+y_0y_s^2+y_0y_{s^2}^2)+x_1(y_0^2y_1+y_1^3+y_1y_s^2+y_1y_{s^2}^2)\\
		&~+x_s(y_0^2y_s+y_1^2y_s+y_s^3+y_sy_{s^2}^2)+x_{s^2}(y_0^2y_{s^2}+y_1^2y_{s^2}+y_s^2y_{s^2}+y_{s^2}^3).
	\end{align*}
	There exist simple $4$-colored $1$-designs with the following parameters:
	Six designs with parameters $1$-$(4,(n(0),n(1),n(s),n(s^2)),2)$ where $(n(0),n(1),n(s),n(s^2))$ is equal to
	$(2,2,0,0)$ or any one of its four permutations.
	The space of Jacobi polynomials $\CJ_{C_{4}^{\text{IV}},T}$ with $|T|=2$ may be generated by the two polynomials
	\begin{align*}
		\CJ^1_{C_{4}^{\text{IV}},2}&=(x_0^2+x_1^2+x_s^2+x_{s^2}^2)(y_0^2+y_1^2+y_s^2+y_{s^2}^2),\\
		\CJ^2_{C_{4}^{\text{IV}},2}&=x_0^2y_0^2+2x_0x_1y_0y_1+2x_0x_sy_0y_s+2x_0x_{s^2}y_0y_{s^2}+x_1^2y_1^2+2x_1x_sy_1y_s\\
		&~+2x_1x_{s^2}y_1y_{s^2}+x_s^2y_s^2+2x_sx_{s^2}y_sy_{s^2}+x_{s^2}^2y_{s^2}^2.
	\end{align*}
	Combining these two equations we obtain $2$-designs with parameters
	\[
	2\text{-}(4,(2,2,0,0),(\lambda_1(P),(\lambda_2(P))),
	\]
	where $\lambda(P)$'s are shown in Table~\ref{Tab:Type_IV_length4}.
	By the coefficient of the term $y_0^2y_1^2$ in the complete weight enumerator of the code, we obtain an upper (resp. lower) bound of $D_{\lambda_{\max}(P)}(4,(2,2,0,0),2)$ (resp. $C_{\lambda_{\max}(P)}(4,(2,2,0,0),2)$).
	\[
		D_{\lambda_{\max}(P)}(4,(2,2,0,0),2)
		\leq 2 \leq 
		C_{\lambda_{\min}(P)}(4,(2,2,0,0),2).
	\]
	The $\lambda_{\max}(P)$'s (resp. $\lambda_{\min}(P)$'s) are 
	shown in Table~\ref{Tab:Type_IV_length4}.
	Similarly, we can obtain an upper (resp. lower) bound of $D$ (resp. $C$) 
	$4$-colored $2$-designs where $(n(0),n(1),n(s),n(s^2))$ is equal to anyone of
	the four choices: 
	$(2,0,2,0)$, $(2,0,0,2)$, $(0,2,2,0)$, $(0,0,2,2)$.
	\begin{table}[h]
		\centering
		\caption{$\lambda$'s in $4$-colored $2$-designs}
		\label{Tab:Type_IV_length4}
		\begin{tabular}{|c||c|c|c|c|c|c|c|c|c|c|}
			\hline
			$P$ & $00$ & $01$ & $0s$ & $0s^2$ & $11$ & $1s$ & $1s^2$ & $ss$ & $ss^2$ & $s^2s^2$ \\
			\hline
			$\lambda_{1}(P)$ & $1$ & $0$ & $0$ & $0$ & $1$ & $0$ & $0$ & $0$ & $0$ & $0$ \\
			\hline
			$\lambda_{2}(P)$ & $0$ & $2$ & $0$ & $0$ & $0$ & $0$ & $0$ & $0$ & $0$ & $0$ \\
			\hline
			$\lambda_{\max}(P)$ & $1$ & $2$ & $0$ & $0$ & $1$ & $0$ & $0$ & $0$ & $0$ & $0$ \\
			\hline
			$\lambda_{\min}(P)$ & $0$ & $0$ & $0$ & $0$ & $0$ & $0$ & $0$ & $0$ & $0$ & $0$ \\
			\hline
		\end{tabular}
	\end{table}

\end{ex}

\subsubsection{Length $6$}

\begin{ex}
	Let $C_{6}^{\text{IV}}$ be the second Hermitian self-dual code over $\FF_4$ of length $6$ in \cite{HM}. 
	\begin{align*}
		f[6]&=2u^6+2u^5v+3u^4v^2+4u^3v^3+3u^2v^4+2uv^5+2v^6.
	\end{align*}
	
	In this case, it holds the following Lemmas and Theorem.
	
	\begin{lem} \label{lem:24}
		A basis of $M_{2,4}^{\widehat{G_{\rm{IV}}}}$ is obtained by applying $R(f,\widehat{G_{\rm{IV}}})$ with~$f$ running over the monomials
		\[
		x_0^2y_0^4,\;x_0^2y_1^2y_s^2,\;x_0x_1y_0y_1y_s^2.
		\]
	\end{lem}
	
	\begin{lem} \label{lem:33}
		A basis of $M_{3,3}^{\widehat{G_{\rm{IV}}}}$ is obtained by applying 
		$R(f,\widehat{G_{\rm{IV}}})$ with~$f$ running over the monomials
		\[
			x_0^3y_0^3,\;x_0^2x_1y_1y_s^2,\;x_0x_1^2y_0y_s^2,\;x_0x_1x_sy_0y_1y_s.
		\]
	\end{lem}
	
	We need to observe that that specialization is one-to-one on those spaces.
	
	\begin{lem} \label{lem:SS}
		For $\ell=1,2$ we have
		\[
			\dim ({S_{8}} \cdot M_{\ell,6-\ell}^{\widehat{G_{\rm{IV}}}})
			=
			\dim(M_{\ell,6-\ell}^{\widehat{G_{\rm{IV}}}}).
		\]
	\end{lem}
	
	\begin{proof}
		We obtain it by taking the image by ${S_{8}}$ of the preceding bases.
	\end{proof}
	
	\begin{thm}\label{Thm:Main3}
		The codewords of fixed composition in the Hermitian Type~$\IV$ code of 
		length~$6$ hold $4$-colored $2$-designs.
	\end{thm}
	
	\begin{proof}
		We need to show that $\CJ_{C,T}$ does not depend on $T$ for $|T|=2$.
		Since it lives in $M_{2,4}^{\widehat{G_{\rm{IV}}}}$, we can expand it with indeterminate coefficients on the basis given in Lemma~\ref{lem:24}.
		The Hermitian code is 3-homogeneous.
		By specialization and Lemma~\ref{Lem:BSULem3IV} and Lemma~\ref{lem:SS}, we determine these completely by solving a $3\times3$ linear system.
	\end{proof}
	
	Note that we can expand $\CJ_{C,T}$ for $|T|=3$ with indeterminate coefficients on the basis given in Lemma~\ref{lem:33} since it lives in $M_{3,3}^{\widehat{G_{\rm{IV}}}}$.
	But we can't determine these completely by solving a $3\times3$ linear system by specialization and Lemma~\ref{Lem:BSULem3IV} and Lemma~\ref{lem:SS}.
	
	Since the codewords of fixed composition in $C_{6}^{\text{IV}}$ holds $4$-colored $2$-design, we assume that
	$|T|=1,2$. Then
	\begin{align*}
		\CJ_{C_{6}^{\text{IV}},1}&=x_0(y_0^5+5y_0y_1^2y_s^2+5y_0y_1^2y_{s^2}^2+5y_0y_s^2y_{s^2}^2)+x_1(5y_0^2y_1y_s^2\\
		&~+5y_0^2y_1y_{s^2}^2+y_1^5+5y_1y_s^2y_{s^2}^2)+x_s(5y_0^2y_1^2y_s+5y_0^2y_sy_{s^2}^2\\
		&~+5y_1^2y_sy_{s^2}^2+y_s^5)+x_{s^2}(5y_0^2y_1^2y_{s^2}+5y_0^2y_s^2y_{s^2}+5y_1^2y_s^2y_{s^2}+y_{s^2}^5),\\
		\CJ_{C_{6}^{\text{IV}},2}&=x_0^2(y_0^4+y_1^2y_s^2+y_1^2y_{s^2}^2+y_s^2y_{s^2}^2)+x_0x_1(4y_0y_1y_s^2+4y_0y_1y_{s^2}^2)\\
		&~+x_0x_s(4y_0y_1^2y_s+4y_0y_sy_{s^2}^2)+x_0x_{s^2}(4y_0y_1y_{s^2}^2+4y_0y_s^2y_{s^2})\\
		&~+x_1^2(y_0^2y_s^2+y_0^2y_{s^2}^2+y_1^4+y_s^2y_{s^2}^2)+x_1x_s(4y_0^2y_1y_s+4y_1y_sy_{s^2}^2)\\
		&~+x_1x_{s^2}(4y_0^2y_1y_{s^2}+y_1y_s^2y_{s^2})+x_s^2(y_0^2y_1^2+y_0^2y_{s^2}^2+y_1^2y_{s^2}^2+y_s^4)\\
		&~+x_sx_{s^2}(4y_0^2y_sy_{s^2}+4y_1^2y_sy_{s^2})+x_{s^2}^2(y_0^2y_1^2+y_0^2y_s^2+y_1^2y_s^2+y_{s^2}^4).
	\end{align*}
	
	\begin{cor}
		Let $\FF_4=\{0,1,s,s^2\}$.
		There exist simple $4$-colored $2$-designs with the following parameters:
		Four designs with parameters $2$-$(6,(n(0),n(1),n(s),n(s^2)),15)$ where $(n(0),n(1),n(s),n(s^2))$ is equal to
		$(2,2,2,0)$ or any one of its four permutations.
	\end{cor}
	
	The space of Jacobi polynomials $\CJ_{C_{6}^{\text{IV}},T}$ with $|T|=3$ may be generated by the two polynomials
	\begin{align*}
		\CJ^1_{C_{6}^{\text{IV}},3}&=x_0^3y_0^3+3x_0^2x_1y_1y_s^2+3x_0^2x_sy_sy_{s^2}^2+3x_0^2x_{s^2}y_1^2y_{s^2}+3x_0x_1^2y_0y_{s^2}^2\\
		&~+6x_0x_1x_sy_0y_1y_s+6x_0x_1x_{s^2}y_0y_1y_{s^2}+3x_0x_s^2y_0y_1^2+6x_0x_sx_{s^2}y_0y_sy_{s^2}\\
		&~+3x_0x_{s^2}^2y_0y_s^2+x_1^3y_1^3+3x_1^2x_sy_0^2y_s+3x_1^2x_{s^2}y_s^2y_{s^2}+3x_1x_s^2y_1y_{s^2}^2\\
		&~+6x_1x_sx_{s^2}y_1y_sy_{s^2}+3x_1x_{s^2}^2y_0^2y_1+x_s^3y_s^3+3x_s^2x_{s^2}y_0^2y_{s^2}\\
		&~+3x_sx_{s^2}^2y_1^2y_s+x_{s^2}^3y_{s^2}^3,\\
		\CJ^2_{C_{6}^{\text{IV}},3}&=x_0^3y_0^3+3x_0^2x_1y_1y_{s^2}^2+3x_0^2x_sy_1^2y_s+3x_0^2x_{s^2}y_s^2y_{s^2}+3x_0x_1^2y_0y_s^2\\
		&~+6x_0x_1x_sy_0y_1y_s+6x_0x_1x_{s^2}y_0y_1y_{s^2}+3x_0x_s^2y_0y_{s^2}^2+6x_0x_sx_{s^2}y_0y_sy_{s^2}\\
		&~+3x_0x_{s^2}^2y_0y_1^2+x_1^3y_1^3+3x_1^2x_sy_sy_{s^2}^2+3x_1^2x_{s^2}y_0^2y_{s^2}+3x_1x_s^2y_0^2y_1\\
		&~+6x_1x_sx_{s^2}y_1y_sy_{s^2}+3x_1x_{s^2}^2y_1y_s^2+x_s^3y_s^3+3x_s^2x_{s^2}y_1^2y_{s^2}+3x_sx_{s^2}^2y_0^2y_s\\
		&~+x_{s^2}^3y_{s^2}^3.
	\end{align*}
	Combining these two equations we obtain $3$-designs with parameters
	\[
	3\text{-}(6,(2,2,2,0),(\lambda_1(P),(\lambda_2(P))),
	\]
	where $\lambda(P)$'s are shown in Table~\ref{Tab:Type_IV_length6}.
	By the coefficient of the term $y_0^2y_1^2y_s^2$ in the complete weight enumerator of the code, we obtain an upper (resp. lower) bound of $D_{\lambda_{\max}(P)}(6,(2,2,2,0),3)$ (resp. $C_{\lambda_{\max}(P)}(6,(2,2,2,0),3)$).
	\[
		D_{\lambda_{\max}(P)}(6,(2,2,2,0),3)
		\leq 15 \leq 
		C_{\lambda_{\min}(P)}(6,(2,2,2,0),3).
	\]
		The $\lambda_{\max}(P)$'s (resp. $\lambda_{\min}(P)$'s) are 
	shown in Table~\ref{Tab:Type_IV_length6}.
	Similarly, we can obtain an upper (resp. lower) bound of $D$ (resp. $C$) 
	$4$-colored $3$-designs where $(n(0),n(1),n(s),n(s^2))$ is equal to anyone of
	the three choices: 
	$(2,2,0,2)$, $(2,0,2,2)$, $(0,2,2,2)$. 
	
	\begin{table}[h]
		\centering
		\caption{$\lambda$'s in $4$-colored $3$-designs}
		\label{Tab:Type_IV_length6}
		\begin{tabular}{|c||c|c|c|c|c|c|c|c|c|c|}
			\hline
			$P$ & $000$ & $001$ & $00s$ & $00s^2$ & $011$ & $01s$ & $01s^2$ & $0ss$ & $0ss^2$ & $0s^2s^2$ \\
			\hline
			$\lambda_{1}(P)$ & $0$ & $0$ & $3$ & $0$ & $3$ & $6$ & $0$ & $0$ & $0$ & $0$ \\
			\hline
			$\lambda_{2}(P)$ & $0$ & $3$ & $0$ & $0$ & $0$ & $6$ & $0$ & $3$ & $0$ & $0$ \\
			\hline
			$\lambda_{\max}(P)$ & $0$ & $3$ & $3$ & $0$ & $3$ & $6$ & $0$ & $3$ & $0$ & $0$ \\
			\hline
			$\lambda_{\min}(P)$ & $0$ & $0$ & $0$ & $0$ & $0$ & $6$ & $0$ & $0$ & $0$ & $0$ \\
			\hline
			$P$ & $111$ & $11s$ & $11s^2$ & $1ss$ & $1ss^2$ & $1s^2s^2$ & $sss$ & $sss^2$ & $ss^2s^2$ & $s^2s^2s^2$ \\
			\hline
			$\lambda_{1}(P)$ & $0$ & $0$ & $0$ & $3$ & $0$ & $0$ & $0$ & $0$ & $0$ & $0$ \\
			\hline
			$\lambda_{2}(P)$ & $0$ & $3$ & $0$ & $0$ & $0$ & $0$ & $0$ & $0$ & $0$ & $0$ \\
			\hline
			$\lambda_{\max}(P)$ & $0$ & $3$ & $0$ & $3$ & $0$ & $0$ & $0$ & $0$ & $0$ & $0$ \\
			\hline
			$\lambda_{\min}(P)$ & $0$ & $0$ & $0$ & $0$ & $0$ & $0$ & $0$ & $0$ & $0$ & $0$ \\
			\hline
		\end{tabular}
	\end{table}

\end{ex}

\subsubsection{Length $8$}

\begin{ex}
	Let $C_{8}^{\text{IV}}$ be the third Hermitian self-dual code over $\FF_4$ of length $8$ in \cite{HM}. 
	\begin{align*}
		f[8]&=3u^8+5u^7v+7u^6v^2+8u^5v^3+10u^4v^4+\cdots.
	\end{align*}
	Since the codewords of fixed composition in $C_{8}^{\text{IV}}$ holds $4$-colored $3$-design, we assume that
	$|T|=1,2,3$.
	Then,
	\begin{align*}
		\CJ_{C_{8}^{\text{IV}},1}&=x_0(y_0^7+7y_0^3y_1^4+7y_0^3y_s^4+7y_0^3y_{s^2}^4+42y_0y_1^2y_s^2y_{s^2}^2)\\
		&~+x_1(7y_0^4y_1^3+42y_0^2y_1y_s^2y_{s^2}^2+y_1^7+7y_1^3y_s^4+7y_1^3y_{s^2}^4)\\
		&~+x_s(7y_0^4y_s^3+42y_0^2y_1^2y_sy_{s^2}^2+7y_1^4y_s^3+y_s^7+7y_s^3y_{s^2}^4)\\
		&~+x_{s^2}(7y_0^4y_{s^2}^3+42y_0^2y_1^2y_s^2y_{s^2}+7y_1^4y_{s^2}^3+7y_s^4y_{s^2}^3+y_{s^2}^7),\\
		\CJ_{C_{8}^{\text{IV}},2}&=x_0^2(y_0^6+3y_0^2y_1^4+3y_0^2y_s^4+3y_0^2y_{s^2}^4+6y_1^2y_s^2y_{s^2}^2)\\
		&~+x_0x_1(8y_0^3y_1^3+24y_0y_1y_s^2y_{s^2}^2)+x_0x_s(8y_0^3y_s^3+24y_0y_1^2y_sy_{s^2}^2)\\
		&~+x_0x_{s^2}(8y_0^3y_{s^2}^3+24y_0y_1^2y_s^2y_{s^2})+x_1^2(3y_0^4y_1^2+6y_0^2y_s^2y_{s^2}^2\\
		&~+y_1^6+3y_1^2y_s^4+3y_1^2y_{s^2}^4)+x_1x_s(24y_0^2y_1y_sy_{s^2}^2+8y_1^3y_s^3)\\
		&~+x_1x_{s^2}(24y_0^2y_1y_s^2y_{s^2}+8y_1^3y_{s^2}^3)+x_s^2(3y_0^4y_s^2+6y_0^2y_1^2y_{s^2}^2\\
		&~+3y_1^4y_s^2+3y_s^2y_{s^2}^4+y_s^6)+x_sx_{s^2}(24y_0^2y_1^2y_sy_{s^2}+8y_s^3y_{s^2}^3)\\
		&~+x_{s^2}^2(3y_0^4y_{s^2}^2+6y_0^2y_1^2y_s^2+3y_1^4y_{s^2}^2+3y_s^4y_{s^2}^2+y_{s^2}^6),\\
		\CJ_{C_{8}^{\text{IV}},3}&=x_0^3(y_0^5+y_0y_s^4+y_0y_s^4+y_0y_{s^2}^4)+x_0^2x_1(6y_0^2y_1^3+6y_1y_s^2y_{s^2}^2)\\
		&~+x_0^2x_s(6y_0^2y_s^3+6y_1^2y_sy_{s^2}^2)+x_0^2x_{s^2}(6y_0^2y_{s^2}^3+6y_1^2y_s^2y_{s^2})\\
		&~+x_0x_1^2(6y_0^3y_1^2+6y_0y_s^2y_{s^2}^2)+24x_0x_1x_sy_0y_1y_sy_{s^2}^2\\
		&~+24x_0x_1x_{s^2}y_0y_1y_s^2y_{s^2}+x_0x_s^2(6y_0^3y_s^2+6y_0y_1^2y_{s^2}^2)\\
		&~+24x_0x_sx_{s^2}y_0y_1^2y_sy_{s^2}+x_0x_{s^2}^2(6y_0^3y_{s^2}^2+6y_0y_1^2y_s^2)\\
		&~+x_1^3(y_0^4y_1+y_1^5+y_1y_s^4+y_1y_{s^2}^4)+x_1^2x_s(6y_0^2y_sy_{s^2}^2+6y_1^2y_s^3)\\
		&~+x_1^2x_{s^2}(6y_0^2y_s^2y_{s^2}+6y_1^2y_{s^2}^3)+x_1x_s^2(6y_0^2y_1y_{s^2}^2+6y_1^3y_s^2)\\
		&~+x_1x_{s^2}^2(6y_0^2y_1y_s^2+6y_1^3y_{s^2}^2)+24x_1x_sx_{s^2}y_0^2y_1y_sy_{s^2}\\
		&~+x_s^3(y_0^4y_s+y_1^4y_s+y_s^5+y_sy_{s^2}^4)+x_s^2x_{s^2}(6y_0^2y_1^2y_{s^2}+6y_s^2y_{s^2}^3)\\
		&~+x_sx_{s^2}^2(6y_0^2y_1^2y_s+6y_s^3y_{s^2}^2)+x_{s^2}^3(y_0^4y_{s^2}+y_1^4y_{s^2}+y_s^4y_{s^2}+y_{s^2}^5).
	\end{align*}
	There exist simple $4$-colored $3$-designs with the following parameters:
	Six designs with parameters $3$-$(8,(n(0),n(1),n(s),n(s^2)),14)$ where $(n(0),n(1),n(s),n(s^2))$ is equal to
	$(4,4,0,0)$ or any one of its four permutations.
	The space of Jacobi polynomials $\CJ_{C_{8}^{\text{IV}},T}$ with $|T|=4$ may be generated by the two polynomials
	\begin{align*}
		\CJ^1_{C_{8}^{\text{IV}},4}&=x_0^4(y_0^4+y_1^4+y_s^4+y_{s^2}^4)+x_0^2x_1^2(12y_0^2y_1^2+12y_s^2y_{s^2}^2)\\
		&~+x_0^2x_s^2(12y_0^2y_s^2+12y_1^2y_{s^2}^2)+x_0^2x_{s^2}^2(12y_0^2y_{s^2}^2+12y_1^2y_s^2)\\
		&~+96x_0x_1x_sx_{s^2}y_0y_1y_sy_{s^2}+x_1^4(y_0^4+y_1^4+y_s^4+y_{s^2}^4)+x_1^2x_s^2(12y_0^2y_{s^2}^2\\
		&~+12y_1^2y_s^2)+x_1^2x_{s^2}^2(12y_0^2y_s^2+12y_1^2y_{s^2}^2)+x_s^4(y_0^4+y_1^4+y_s^4+y_{s^2}^4)\\
		&~+x_s^2x_{s^2}^2(12y_0^2y_1^2+12y_s^2y_{s^2}^2)+x_{s^2}^4(y_0^4+y_1^4+y_s^4+y_{s^2}^4),\\
		\CJ^2_{C_{8}^{\text{IV}},4}&=x_0^4y_0^4+4x_0^3x_1y_0y_1^3+4x_0^3x_sy_0y_s^3+4x_0^3x_{s^2}y_0y_{s^2}^3+6x_0^2x_1^2y_0^2y_1^2\\
		&~+12x_0^2x_1x_sy_1y_sy_{s^2}^2+12x_0^2x_1x_{s^2}y_1y_s^2y_{s^2}+6x_0^2x_s^2y_0^2y_s^2\\
		&~+12x_0^2x_sx_{s^2}y_1^2y_sy_{s^2}+6x_0^2x_{s^2}^2y_0^2y_{s^2}^2+4x_0x_1^3y_0^3y_1+12x_0x_1^2x_sy_0y_sy_{s^2}^2\\
		&~+12x_0x_1^2x_{s^2}y_0y_s^2y_{s^2}+12x_0x_1x_s^2y_0y_1y_{s^2}^2+24x_0x_1x_sx_{s^2}y_0y_1y_sy_{s^2}\\
		&~+12x_0x_1x_{s^2}^2y_0y_1y_s^2+4x_0x_s^3y_0^3y_s+12x_0x_s^2x_{s^2}y_0y_1^2y_{s^2}\\
		&~+12x_0x_sx_{s^2}^2y_0y_1^2y_s+4x_0x_{s^2}^3y_0^3y_{s^2}+x_1^4y_1^4+4x_1^3x_sy_1y_s^3\\
		&~+4x_1^3x_{s^2}y_1y_{s^2}^3+6x_1^2x_s^2y_1^2y_s^2+12x_1^2x_sx_{s^2}y_0^2y_sy_{s^2}+6x_1^2x_{s^2}^2y_1^2y_{s^2}^2\\
		&~+4x_1x_s^3y_1^3y_s+12x_1x_s^2x_{s^2}y_0^2y_1y_{s^2}+12x_1x_sx_{s^2}^2y_0^2y_1y_s+4x_1x_{s^2}^3y_1^3y_{s^2}\\
		&~+x_s^4y_s^4+4x_s^3x_{s^2}y_sy_{s^2}^3+6x_s^2x_{s^2}^2y_s^2y_{s^2}^2+4x_sx_{s^2}^3y_s^3y_{s^2}+x_{s^2}^4y_{s^2}^4.
	\end{align*}
	Combining these two equations we obtain $4$-designs with parameters
	\[
		4\text{-}(8,(4,4,0,0),(\lambda_{1}^{1}(P),\lambda_{2}^{1}(P)))
		\text{ and }
		4\text{-}(8,(2,2,2,2),(\lambda_{1}^{2}(P),\lambda_{2}^{2}(P)),
	\]
	where $\lambda(P)$'s are shown in Table~\ref{Tab:Type_IV_length8}.
	By the coefficient of the term $y_0^{n(0)}y_1^{n(1)}y_s^{n(s)}y_{s^2}^{n(s^2)}$ in the complete weight enumerator of the code, we obtain an upper (resp. lower) bound of $D_{\lambda_{\max}(P)}(8,(n(0),n(1),n(s),n(s^2)),4)$ (resp. $C_{\lambda_{\max}(P)}(8,(n(0),n(1),n(s),n(s^2)),4)$).
	\begin{align*}
		D_{\lambda_{\max}^{1}(P)}(8,(4,4,0,0),4)
		&\leq 14 \leq 
		C_{\lambda_{\min}^{1}(P)}(8,(4,4,0,0),4),\\
		D_{\lambda_{\max}^{2}(P)}(8,(2,2,2,2),4)
		&\leq 168 \leq 
		C_{\lambda_{\min}^{2}(P)}(8,(2,2,2,2),4).
	\end{align*}
	The $\lambda_{\max}(P)$'s (resp. $\lambda_{\min}(P)$'s) are 
	shown in Table~\ref{Tab:Type_IV_length8}. 
	Similarly, we can obtain an upper (resp. lower) bound of $D$ (resp. $C$) 
	$4$-colored $3$-designs where $(n(0),n(1),n(s),n(s^2))$ is equal to anyone of
	the five choices: 
	$(4,0,4,0)$, $(4,0,0,4)$, $(0,4,4,0)$, $(0,4,0,4)$, $(0,0,4,4)$.
	\small
	\begin{table}[h]
		\centering
		\caption{$\lambda$'s in $4$-colored $4$-designs}
		\label{Tab:Type_IV_length8}
		\begin{tabular}{|c||c|c|c|c|c|c|c|}
			\hline
			$P$ & $0000$ & $0001$ & $000s$ & $000s^2$ & $0011$ & $001s$ & $001s^2$ \\
			\hline
			$\lambda_{1}^{1}(P)$ & $1$ & $0$ & $0$ & $0$ & $12$ & $0$ & $0$ \\
			\hline
			$\lambda_{2}^{1}(P)$ & $0$ & $4$ & $0$ & $0$ & $6$ & $0$ & $0$ \\
			\hline
			$\lambda_{\max}^{1}(P)$ & $1$ & $4$ & $0$ & $0$ & $12$ & $0$ & $0$ \\
			\hline
			$\lambda_{\min}^{1}(P)$ & $0$ & $0$ & $0$ & $0$ & $6$ & $0$ & $0$ \\
			\hline
			$\lambda_{1}^{2}(P)$ & $0$ & $0$ & $0$ & $0$ & $12$ & $0$ & $0$ \\
			\hline
			$\lambda_{2}^{2}(P)$ & $0$ & $0$ & $0$ & $0$ & $0$ & $12$ & $12$ \\
			\hline
			$\lambda_{\max}^{2}(P)$ & $0$ & $0$ & $0$ & $0$ & $12$ & $12$ & $12$ \\
			\hline
			$\lambda_{\min}^{2}(P)$ & $0$ & $0$ & $0$ & $0$ & $0$ & $0$ & $0$ \\
			\hline
			$P$ & $00ss$ & $00ss^2$ & $00s^2s^2$ & $0111$ & $011s$ & $011s^2$ & $01ss$ \\
			\hline
			$\lambda_{1}^{1}(P)$ & $0$ & $0$ & $0$ & $0$ & $0$ & $0$ & $0$ \\
			\hline
			$\lambda_{2}^{1}(P)$ & $0$ & $0$ & $0$ & $4$ & $0$ & $0$ & $0$ \\
			\hline
			$\lambda_{\max}^{1}(P)$ & $0$ & $0$ & $0$ & $4$ & $0$ & $0$ & $0$ \\
			\hline
			$\lambda_{\min}^{1}(P)$ & $0$ & $0$ & $0$ & $0$ & $0$ & $0$ & $0$ \\
			\hline
			$\lambda_{1}^{2}(P)$ & $12$ & $0$ & $12$ & $0$ & $0$ & $0$ & $0$ \\
			\hline
			$\lambda_{2}^{2}(P)$ & $0$ & $12$ & $0$ & $0$ & $12$ & $12$ & $12$ \\
			\hline
			$\lambda_{\max}^{2}(P)$ & $12$ & $12$ & $12$ & $0$ & $12$ & $12$ & $12$ \\
			\hline
			$\lambda_{\min}^{2}(P)$ & $0$ & $0$ & $0$ & $0$ & $0$ & $0$ & $0$ \\
			\hline
			$P$ & $01ss^2$ & $01s^2s^2$ & $0sss$ & $0sss^2$ & $0ss^2s^2$ & $0s^2s^2s^2$ & $1111$ \\
			\hline
			$\lambda_{1}^{1}(P)$ & $0$ & $0$ & $0$ & $0$ & $0$ & $0$ & $1$ \\
			\hline
			$\lambda_{2}^{1}(P)$ & $0$ & $0$ & $0$ & $0$ & $0$ & $0$ & $0$ \\
			\hline
			$\lambda_{\max}^{1}(P)$ & $0$ & $0$ & $0$ & $0$ & $0$ & $0$ & $1$ \\
			\hline
			$\lambda_{\min}^{1}(P)$ & $0$ & $0$ & $0$ & $0$ & $0$ & $0$ & $0$ \\
			\hline
			$\lambda_{1}^{2}(P)$ & $96$ & $0$ & $0$ & $0$ & $0$ & $0$ & $0$ \\
			\hline
			$\lambda_{2}^{2}(P)$ & $24$ & $12$ & $0$ & $12$ & $12$ & $0$ & $0$ \\
			\hline
			$\lambda_{\max}^{2}(P)$ & $96$ & $12$ & $0$ & $12$ & $12$ & $0$ & $0$ \\
			\hline
			$\lambda_{\min}^{2}(P)$ & $24$ & $0$ & $0$ & $0$ & $0$ & $0$ & $0$ \\
			\hline
			$P$ & $111s$ & $111s^2$ & $11ss$ & $11ss^2$ & $11s^2s^2$ & $1sss$ & $1sss^2$ \\
			\hline
			$\lambda_{1}^{1}(P)$ & $0$ & $0$ & $0$ & $0$ & $0$ & $0$ & $0$ \\
			\hline
			$\lambda_{2}^{1}(P)$ & $0$ & $0$ & $0$ & $0$ & $0$ & $0$ & $0$ \\
			\hline
			$\lambda_{\max}^{1}(P)$ & $0$ & $0$ & $0$ & $0$ & $0$ & $0$ & $0$ \\
			\hline
			$\lambda_{\min}^{1}(P)$ & $0$ & $0$ & $0$ & $0$ & $0$ & $0$ & $0$ \\
			\hline
			$\lambda_{1}^{2}(P)$ & $0$ & $0$ & $12$ & $0$ & $12$ & $0$ & $0$ \\
			\hline
			$\lambda_{2}^{2}(P)$ & $12$ & $0$ & $0$ & $12$ & $0$ & $0$ & $12$ \\
			\hline
			$\lambda_{\max}^{2}(P)$ & $12$ & $0$ & $12$ & $12$ & $12$ & $0$ & $12$ \\
			\hline
			$\lambda_{\min}^{2}(P)$ & $0$ & $0$ & $0$ & $0$ & $0$ & $0$ & $0$ \\
			\hline
			
			$P$ & $1ss^2s^2$ & $1s^2s^2s^2$ & $ssss$ & $ssss^2$ & $sss^2s^2$ & $ss^2s^2s^2$ & $s^2s^2s^2s^2$ \\
			\hline
			$\lambda_{1}^{1}(P)$ & $0$ & $0$ & $0$ & $0$ & $0$ & $0$ & $0$ \\
			\hline
			$\lambda_{2}^{1}(P)$ & $0$ & $0$ & $0$ & $0$ & $0$ & $0$ & $0$ \\
			\hline
			$\lambda_{\max}^{1}(P)$ & $0$ & $0$ & $0$ & $0$ & $0$ & $0$ & $0$ \\
			\hline
			$\lambda_{\min}^{1}(P)$ & $0$ & $0$ & $0$ & $0$ & $0$ & $0$ & $0$ \\
			\hline
			$\lambda_{1}^{2}(P)$ & $0$ & $0$ & $0$ & $0$ & $12$ & $0$ & $0$ \\
			\hline
			$\lambda_{2}^{2}(P)$ & $12$ & $0$ & $0$ & $0$ & $0$ & $0$ & $0$ \\
			\hline
			$\lambda_{\max}^{2}(P)$ & $12$ & $0$ & $0$ & $0$ & $12$ & $0$ & $0$ \\
			\hline
			$\lambda_{\min}^{2}(P)$ & $0$ & $0$ & $0$ & $0$ & $0$ & $0$ & $0$ \\
			\hline
		\end{tabular}
	\end{table}
	\normalsize

\end{ex}

In the case the extremal Type III (resp. Type IV) code of length $n$ containing the all-one vector, it holds  $3$- (resp. $4$-) colored $t$-design as in~Table~\ref{Tab:Type_III_colored_designs} 
(resp. Table~\ref{Tab:Type_IV_colored_designs}).

\begin{table}[h]
	\centering
	\caption{$3$-colored $t$-$(n,(n_0,n_1,n_2),|\mathcal{B}|)$ design in Type III code}
	\label{Tab:Type_III_colored_designs}
	\begin{tabular}{c|c|c|c}
		\hline
		$n$ & $t$ & \shortstack{Blocks in the cwe $(n_0,n_1,n_2)$\\ up to permutation} & Number of blocks \\
		\hline
		12&1&$(6,3,3)$&220\\
		\hline
		&1&$(6,6,0)$&22\\
		\hline
		&2&$(6,3,3)$&220\\
		\hline
		&2&$(6,6,0)$&22\\
		\hline
		&3&$(6,3,3)$&220\\
		\hline
		&3&$(6,6,0)$&22\\
		\hline
	\end{tabular}
\end{table}

\begin{table}[h]
	\centering
	\caption{$4$-colored $t$-$(n,(n_0,n_1,n_s,n_{s^2}),|\mathcal{B}|)$ design in Type IV code}
	\label{Tab:Type_IV_colored_designs}
	\begin{tabular}{c|c|c|c}
		\hline
		$n$&$t$&\shortstack{Blocks in the cwe $(n_0,n_1,n_s,n_{s^2})$\\ up to permutation}&Number of blocks \\
		\hline
		4&1&$(2,2,0,0)$&2\\
		\hline
		6&1&$(2,2,2,0)$&15\\
		\hline
		8&1&$(4,4,0,0)$&14\\
		\hline
		8&1&$(2,2,2,2)$&168\\
		\hline
		6&2&$(2,2,2,0)$&15\\
		\hline
		8&2&$(4,4,0,0)$&14\\
		\hline
		8&2&$(2,2,2,2)$&168\\
		\hline
		8&3&$(4,4,0,0)$&14\\
		\hline
		8&3&$(2,2,2,2)$&168\\
		\hline
	\end{tabular}
\end{table}

\section{Concluding remarks}\label{Sec:Conclusion}



Let $D_{w}$ be the support design of a code $C$ for weight $w$ and
\begin{align*}
	\delta(C)&:=\max\{t\in \mathbb{N}\mid \forall w,
	D_{w} \mbox{ is a } t\mbox{-design}\},\\
	s(C)&:=\max\{t\in \mathbb{N}\mid \exists w \mbox{ s.t.~}
	D_{w} \mbox{ is a } t\mbox{-design}\}.
\end{align*}
We note that $\delta(C) \leq s(C)$.
In our previous papers
\cite{BMN, extremal design H-M-N, MMN, extremal design2 M-N,MN-tec, dual support designs,{mn-typeI}},
we considered the possible occurrence of $\delta(C)<s(C)$.
This was motivated by Lehmer's conjecture, which is
an analogue of $\delta(C)<s(C)$ in the theory of lattices
and vertex operator algebras.
For the details,
see \cite{{BM1},{BM2},{BMY},{Lehmer},{Miezaki},{Miezaki2},{MMN},{Venkov},{Venkov2}}.

Let $CD_{w}$ be the support colored design of a code $C$ for weight $w$ and
\begin{align*}
	\delta_c(C)&:=\max\{t\in \mathbb{N}\mid \forall w,
	CD_{w} \mbox{ is a colored } t\mbox{-design}\},\\
	s_c(C)&:=\max\{t\in \mathbb{N}\mid \exists w \mbox{ s.t.~}
	CD_{w} \mbox{ is a colored } t\mbox{-design}\}.
\end{align*}
It is natural to give upper and lower bounds
of $\delta_c(C)$ and $s_c(C)$ for all extremal Type II, III, and IV codes.

We will continue the study of this paper in~\cite{CIMT20xx} to the case 
of $\ZZ_{k}$-codes as a generalization of the works done by
Bonnecaze et al.~\cite{BMS1999}.
Moreover, we investigate the colored designs to the case of
Kleinian codes in~\cite{CIMTxxxx}.

\section*{Declaration of competing interest}

The authors declare that they have no known competing financial interests or personal relationships that could have appeared to influence the work reported in this paper.

\section*{Acknowledgements}


The authors thank Tsuyoshi Miezaki and Manabu Oura for their helpful discussions 
and comments to this research.


\section*{Data availability statement}

The data that support the findings of this study are available from
the corresponding author.

\end{document}